\documentclass[a4paper,11pt,fleqn]{article}
\usepackage{amsmath}
\usepackage{amssymb}
\usepackage{mathabx} 
\usepackage{theorem}
\usepackage{mathrsfs} 
\usepackage[usenames,dvipsnames]{pstricks}
\usepackage{pstricks-add}
\usepackage{euscript}
\usepackage{graphicx}
\usepackage{enumitem}
\usepackage{charter}
\usepackage{empheq}
\usepackage{caption}
\usepackage{subcaption}
\newcommand{\email}[1]{\href{mailto:#1}{\nolinkurl{#1}}}
\definecolor{dred}{HTML}{D90404}
\definecolor{labelkey}{rgb}{0,0.08,0.45}
\definecolor{refkey}{rgb}{0,0.6,0.0}
\definecolor{Brown}{rgb}{0.45,0.0,0.05}
\definecolor{dgreen}{rgb}{0.00,0.49,0.00}
\definecolor{dblue}{rgb}{0,0.08,0.75}
\RequirePackage[colorlinks,hyperindex]{hyperref} 
\hypersetup{linktocpage=true,citecolor=dblue,linkcolor=dgreen}
\PassOptionsToPackage{normalem}{ulem}
\usepackage{ulem}

\newcommand{\setoile}{{\scalebox{0.7}{\ensuremath{*}}}}
\newcommand{\petoile}{p^{\scalebox{0.7}{\ensuremath{*}}}}
\newcommand{\phietoile}{\phi^{\scalebox{0.7}{\ensuremath{*}}}}
\newcommand{\varphie}{\varphi^{\scalebox{0.7}{\ensuremath{*}}}}
\newcommand{\fetoile}{f^{\scalebox{0.7}{\ensuremath{*}}}}
\renewcommand{\leq}{\ensuremath{\leqslant}}
\renewcommand{\geq}{\ensuremath{\geqslant}}

\newcommand{\Scal}[2]{\bigg\langle{#1}\;\bigg|\:{#2}\bigg\rangle} 
\newcommand{\scal}[2]{{\left\langle{{#1}\mid{#2}}\right\rangle}}

\newcommand{\menge}[2]{\big\{{#1}~\big |~{#2}\big\}} 
\newcommand{\mEnge}[2]{\bigg\{{#1}~\big |~{#2}\bigg\}} 

\newcommand{\phaut}[1]{{#1}^{{\tiny\mbox{$\wedge$}}}}
\newcommand{\haut}[1]{{#1}^{{\small\mbox{$\blacktriangleup$}}}}
\newcommand{\haute}[1]{{#1}^{\small\raisebox{-0.2mm}%
{$\blacktriangleup$}*}}
\newcommand{\pbas}[1]{{#1}^{{\tiny\mbox{$\vee$}}}}

\newcommand{\bas}[1]{{#1}^{\small\raisebox{0.0mm}%
{$\blacktriangledown$}}}
\newcommand{\base}[1]{{#1}^{\small\raisebox{0.0mm}%
{$\blacktriangledown$}*}}

\newcommand{\rec}{\ensuremath{\text{\rm rec}\,}}
\newcommand{\epi}{\ensuremath{\text{\rm epi}\,}}

\newcommand{\HH}{\ensuremath{{\mathcal H}}}

\newcommand{\GG}{\ensuremath{{\mathcal G}}}

\newcommand{\rocky}{\ensuremath{\mbox{\footnotesize$\odot$}}}

\newcommand{\emp}{\ensuremath{{\varnothing}}}
\newcommand{\Id}{\ensuremath{\operatorname{Id}}}

\newcommand{\ppersp}{\ensuremath{\ltimes}}
\newcommand{\persp}{\ensuremath{\mbox{\small$\ltimes$}%
\hspace{-2.6mm}\raisebox{0.20mm}%
{\mbox{\scriptsize$\blacktriangleright$}}\;}}

\newcommand{\RR}{\ensuremath{\mathbb{R}}}
\newcommand{\RP}{\ensuremath{\left[0,+\infty\right[}}
\newcommand{\RM}{\ensuremath{\left]-\infty,0\right]}}
\newcommand{\RMX}{\ensuremath{\left[-\infty,0\right]}}
\newcommand{\RMM}{\ensuremath{\left]-\infty,0\right[}}
\newcommand{\RMMX}{\ensuremath{\left[-\infty,0\right[}}

\newcommand{\RPP}{\ensuremath{\left]0,+\infty\right[}}
\newcommand{\RPPX}{\ensuremath{\left]0,+\infty\right]}}

\newcommand{\RPX}{\ensuremath{\left[0,+\infty\right]}}
\newcommand{\RX}{\ensuremath{\left]-\infty,+\infty\right]}}
\newcommand{\RXX}{\ensuremath{\left[-\infty,+\infty\right]}}
\newcommand{\RXM}{\ensuremath{\left[-\infty,+\infty\right[}}

\newcommand{\NN}{\ensuremath{\mathbb N}}

\newcommand{\range}{\ensuremath{\text{\rm range}\,}}

\newcommand{\pinf}{\ensuremath{{+\infty}}}
\newcommand{\minf}{\ensuremath{{-\infty}}}
\newcommand{\dom}{\ensuremath{\text{\rm dom}\,}}
\newcommand{\cdom}{\ensuremath{\overline{\text{\rm dom}}\,}}
\newcommand{\proj}{\ensuremath{\text{\rm proj}}}
\newcommand{\prox}{\ensuremath{\text{\rm prox}}}

\newcommand{\sign}{\ensuremath{\text{\rm sign}}}

\newcommand{\inte}{\ensuremath{\text{\rm int}}\,}

\newcommand{\cconv}{\ensuremath{\overline{\text{\rm conv}}\,}}

\newcommand{\zeroun}{\ensuremath{\left]0,1\right[}}

\newtheorem{theorem}{Theorem}[section]
\newtheorem{lemma}[theorem]{Lemma}

\newtheorem{proposition}[theorem]{Proposition}
\theoremstyle{plain}{\theorembodyfont{\rmfamily}%
}
\theoremstyle{plain}{\theorembodyfont{\rmfamily}%
}
\theoremstyle{plain}{\theorembodyfont{\rmfamily}%
}
\theoremstyle{plain}{\theorembodyfont{\rmfamily}%
}
\theoremstyle{plain}{\theorembodyfont{\rmfamily}%
}
\theoremstyle{plain}{\theorembodyfont{\rmfamily}%
}
\theoremstyle{plain}{\theorembodyfont{\rmfamily}%
\newtheorem{example}[theorem]{Example}}
\theoremstyle{plain}{\theorembodyfont{\rmfamily}%
\newtheorem{remark}[theorem]{Remark}}
\theoremstyle{plain}{\theorembodyfont{\rmfamily}%
\newtheorem{definition}[theorem]{Definition}}
\theoremstyle{plain}{\theorembodyfont{\rmfamily}%
}

\numberwithin{equation}{section}
\def\abstract{\noindent{\bfseries Abstract}. \ignorespaces}

\begin{document}
\title{\sffamily \vskip -9mm 
Proximity Operators of Perspective Functions with Nonlinear 
Scaling\thanks{Contact author: P. L. Combettes, 
\email{plc@math.ncsu.edu}, phone: +1 919 515 2671. 
The work of L. M. Brice\~{n}o-Arias was supported by 
Agencia Nacional de Investigaci\'on y Desarrollo 
(ANID-Chile) under grants FONDECYT 1190871 and 1230257, and by
Centro de Modelamiento Matem\'atico (CMM), ACE210010 and FB210005,
BASAL funds for centers of excellence. 
The work of P. L. Combettes was supported by the National Science
Foundation under grant DMS-1818946. The work of F. J. Silva was
partially supported by Agence Nationale de la Recherche (ANR),
project ANR-22-CE40-0010, and by KAUST through the subaward
agreement ORA-2021-CRG10-4674.6.
}}
\author{Luis M. Brice\~{n}o-Arias$^{1}$, 
Patrick L. Combettes$^2$, 
and Francisco J. Silva$^3$
\\[4mm]
\small
\small $\!^1$Universidad T\'ecnica Federico Santa Mar\'ia,
Departamento de Matem\'atica, Santiago, Chile\\
\small{\email{luis.briceno@usm.cl}}\\[3mm]
\small $\!^2$North Carolina State University,
Department of Mathematics, Raleigh, NC 27695-8205, USA\\
\small{\email{plc@math.ncsu.edu}}\\[4mm]
\small $\!^3$Universit\'e de Limoges, Laboratoire XLIM, 
87060 Limoges, France\\
\small{\email{francisco.silva@unilim.fr}}
}

\bigskip
\date{~}
\maketitle

\begin{abstract}
A perspective function is a construction which combines a base
function defined on a given space with a nonlinear scaling function
defined on another space and which yields a lower semicontinuous
convex function on the product space. Since perspective functions
are typically nonsmooth, their use in first-order algorithms
necessitates the computation of their proximity operator. This
paper establishes closed-form expressions for the proximity
operator of a perspective function defined on a Hilbert space in
terms of a proximity operator involving its base function and one
involving its scaling function.
\end{abstract} 


\newpage
\section{Introduction}
\label{sec:1}

Throughout, $\HH$ and $\GG$ are real Hilbert spaces and
$\Gamma_0(\HH)$ is the class of proper lower semicontinuous convex
functions from $\HH$ to $\RX$. The focus of this paper is on the
following construction which combines a base function defined on a
given space with a nonlinear scaling function defined on another
space and which yields a lower semicontinuous convex function on
the product space (alternative constructions of nonlinearly scaled
perspective functions in certain settings have been studied in
\cite{Mar05a,Mar05b,Zali08}; see \cite{Vina23} for a discussion).

\begin{definition}{\rm\cite{Vina23}}
\label{d:1}
The \emph{preperspective} of a \emph{base} function
$\varphi\colon\HH\to\RXX$ with respect to a \emph{scaling} 
function $s\colon\GG\to\RXX$ is
\begin{equation}
\label{e:1}
\begin{array}{lcll}
\!\!\!\varphi\ppersp s\colon\!\!&\!\!\HH\times\GG&\!\!\to\!\!
&\RXX\\[2mm]
&\!\!(x,y)\!\!&\!\!\mapsto\!\!&
\begin{cases}
s(y)\varphi\biggl(\dfrac{x}{s(y)}\biggr),
&\text{if}\;\:0<s(y)<\pinf;\\
\pinf,&\text{if}\;\;\minf\leq s(y)\leq 0
\;\;\text{or}\;\;s(y)=\pinf,
\end{cases}
\end{array}
\end{equation}
and the \emph{perspective} of $\varphi$ with respect to
$s$ is the largest lower semicontinuous convex function 
$\varphi\persp s$ minorizing $\varphi\ppersp s$.
\end{definition}

If $\varphi\in\Gamma_0(\HH)$, $\GG=\RR$, and $s\colon y\mapsto y$
in Definition~\ref{d:1}, it follows from \cite[Theorem~3.E]{Rock66}
that $\varphi\persp s$ reduces to 
\begin{equation} 
\label{e:0}
\widetilde{\varphi}\colon \HH\times\RR\to\RX\colon (x,y)\mapsto
\begin{cases} 
y\varphi\biggl(\dfrac{x}{y}\biggr),&\text{if}\;\:y>0;\\
(\rec\varphi)(x),&\text{if}\;\:y=0;\\ 
\pinf,&\text{if}\;\:y<0,
\end{cases} 
\end{equation} 
where $\rec\varphi$ denotes the recession function of $\varphi$.
This construction, which features a linear scaling function,
corresponds to the classical notion of a perspective function. It
was first considered in \cite{Rock66} and its properties have been
further investigated in \cite{Svva18,Rock70}. On the application
side, an early occurrence of \eqref{e:0} is found in statistical
inference. In this context, a fundamental objective is the
estimation of both the location $x$ (i.e., the regression vector)
and the scale $y$ (e.g., the standard deviation of the noise or
some other parameter) of the statistical model from the data. In
robust statistics, the maximum likelihood-type estimator
(M-estimator) for location with concomitant scale
\cite[p.~179]{Hube81} couples both parameters via a convex
objective function which is precisely of the form \eqref{e:0}.
Other applications involving linearly scaled perspective functions
can be found in 
\cite{Arav18,Bren00,Jmaa18,Ejst20,Elgh18,Kuro22,Lamb16}.

In recent years, perspective functions with nonlinear scaling 
have appeared implicitly in several applications. Thus, 
in the classical dynamical formulation of optimal transport
problems~\cite{Bren00}, $x$ and $y$ represent, respectively, the
momentum and density variables, and the transport cost is
written in terms of the classical perspective function with linear
scaling \eqref{e:0}. In order to model more complex phenomena
nonlinear scaling functions have also been used. For instance, 
the transport cost used in \cite{Card13,Dolb09} involves the
nonlinear scaling function
\begin{equation}
\label{e:v}
s\colon\RR\to\RXM\colon y\mapsto 
\begin{cases}
y^q,&\text{if}\:\:y\geq 0;\\
\minf,&\text{if}\:\:y<0,
\end{cases}
\end{equation}
where $q\in\zeroun$. This modification makes it possible to take
into account congestion effects during the transport of an initial
distribution towards a target one. On the other hand, 
in multiphase optimal transport problems, the density 
is split in several components. For instance, in the case of two
phases, the model of \cite{Bren02} hinges on the scaling function 
\begin{equation}
\label{e:defsBP}
s\colon\RR\to\RXM\colon {y}\mapsto 
\begin{cases}
\left(\dfrac{\alpha}{y}+\dfrac{\beta}{1-y}\right)^{-1},
&\text{if}\:\:{y}\in\left]0,1\right[;\\[5mm]
\minf,&\text{otherwise},
\end{cases}
\end{equation}
where $\alpha$ and $\beta$ are strictly positive parameters 
representing the contribution of each phase to the kinetic energy 
of the system. In \cite{Carr10,Carr24}, the scaling function 
\begin{equation}
\label{e:jose}
s\colon\RR\to\RXM\colon {y}\mapsto 
\begin{cases}
y(1-y),
&\text{if}\:\:{y}\in[0,1];\\[5mm]
\minf,&\text{otherwise},
\end{cases}
\end{equation}
is employed to model the mobility of particles via a nonlinear function of the
density with saturation in order to avoid overcrowding;
this model arises in chemotaxis, phase
segregation, and thin liquid film problems. Definition~\ref{d:1}
can also be found in  
the family of scaling functions considered in \cite{Maas11}, which
includes the logarithmic mean 
\begin{equation}
\label{e:log_scale}
s\colon\RR^2\to\RXM\colon(y_1,y_2)\mapsto
\begin{cases} 0, & \text{if}\; 
(y_1,y_2)\in(\{0\}\times\RP)\cup(\RPP\times\{0\});\\
y_1, & \text{if}\; y_1=y_2\in\RPP;\\
\dfrac{y_2-y_1}{\log(y_2)-\log(y_1)}, & \text{if}\; 
(y_1,y_2)\in\RPP\times\RPP\;\text{and}\;
y_1\neq y_2;\\
\minf, & \text{otherwise,}
\end{cases} 
\end{equation}
as well as the geometric mean 
\begin{equation}
s\colon\RR^2\to\RXM\colon (y_1,y_2)\mapsto
\begin{cases} 
\sqrt{y_1y_2}, &\text{if}\; (y_1,y_2)\in\RP\times\RP; \\
\minf, & \text{otherwise}.
\end{cases}
\end{equation}
In \cite{Maas11}, these scaling functions are used to define
metrics on the space of probability measures on graphs and to
describe the associated gradient flows. Nonlinear scaling functions
also appear in mean field type control \cite{Ach16a,Ach16b},
machine learning \cite{Aval18}, physics \cite{Berc13}, operator
theory \cite{Carl19,Effr09}, mathematical programming
\cite{Imai88,Mare01}, information theory \cite{Lutw05,Berc21},
kinetic theory \cite{Car22b}, and economics \cite{Zell66}.

A key tool in Hilbertian convex analysis to study variational
problems and design solution algorithms for them is Moreau's
proximity operator \cite{Mor62b,More65}. Recall that, given 
$f\in\Gamma_0(\HH)$ and $x\in\HH$,
\begin{equation}
\label{e:p}
\prox_f\hspace{.2mm}x\;\text{is the unique minimizer over $\HH$
of the function}\;y\mapsto f(y)+\dfrac{1}{2}\|x-y\|^2.
\end{equation}
This process defines the proximity operator
$\prox_f\colon\HH\to\HH$ of $f$, which is extensively discussed
in \cite{Livre1}. From an algorithmic viewpoint, proximity
operators play a critical role in first order methods as they
constitute the main device to activate nonsmooth functions in
convex minimization problems; see
\cite{Livre1,Cham16,Acnu24,Cond23} and the references therein.
Since the nonlinearly scaled perspectives functions of
Definition~\ref{d:1} are typically nonsmooth, knowledge of their
proximity operator is required to solve the variational
formulations in which they are present. 
Formulas for the proximity operator of the 
classical perspective function $\widetilde{\varphi}$ of
\eqref{e:0} were derived in \cite{Jmaa18,Ejst20} and they have been
employed to solve minimization problems arising in areas such as
statistical biosciences \cite{Stat21},
information theory \cite{Elgh18},
signal recovery \cite{Kuro22},
and machine learning \cite{Yama22}. 
For nonlinear scales, an instance of a proximity operator of 
$\varphi\persp s$, where $\varphi=|\cdot|^2$ and $s$ as in
\eqref{e:jose}, is computed partially in \cite{Carr24}.

It is the objective of the present paper to derive closed-form
expressions for the proximity operator of the general nonlinearly
scaled perspective functions of Definition~\ref{d:1}. 
In turn, our results make minimization problems 
involving perspective functions with nonlinear scaling readily amenable
to powerful proximal splitting solution algorithms. The
closed-form expressions we shall obtain for 
$\prox_{\varphi\persp s}$ will be formulated in terms of a
proximity operator involving the base function $\varphi$ and one
involving its scaling function $s$. For instance, 
Example~\ref{ex:4} gives the formula for the proximity operator
of the perspective function of \cite{Card13,Dolb09}, which involves
the scaling function \eqref{e:v}. 

In Section~\ref{sec:2}, we define our notation and provide the
background necessary to our investigation. Section~\ref{sec:3} is
devoted to preliminary results. Closed-form expressions of
$\prox_{\varphi\persp s}$ are established in Section~\ref{sec:4}
and Examples are provided in Section~\ref{sec:5}.

\section{Notation and background}
\label{sec:2}

The scalar product of a Hilbert space is denoted by
$\scal{\cdot}{\cdot}$ and the associated norm by $\|\cdot\|$.
The closed ball with center $x\in\HH$ and radius $\rho\in\RPP$ is
denoted by $B(x;\rho)$. The Hilbert direct sum of $\HH$ and $\GG$
is denoted by $\HH\oplus\GG$. Let $f\colon\HH\to\RXX$. Then 
$\dom f=\menge{x\in\HH}{f(x)<\pinf}$ is the domain of $f$,
$\epi f=\menge{(x,\xi)\in\HH\times\RR}{f(x)\leq\xi}$ is the
epigraph of $f$, 
\begin{equation}
\label{e:c}
f^*\colon\HH\to\RXX\colon x^*\mapsto\sup_{x\in\HH}
\bigl(\scal{x}{x^*}-f(x)\bigr),
\end{equation}
is the conjugate of $f$, and $\partial f$ is the subdifferential of
$f$. We declare $f$ convex if $\epi f$ is
convex, lower semicontinuous if $\epi f$ is closed, and proper if
$\minf\notin f(\HH)\neq\{\pinf\}$. The recession of 
$f\in\Gamma_0(\HH)$ is 
\begin{equation}
\label{e:rec}
\rec f\colon\HH\to\RXX\colon x\mapsto\lim_{0<\lambda\to\pinf}
\dfrac{f(z+\lambda x)-f(z)}{\lambda},
\end{equation}
where $z\in\dom f$ is arbitrary.
Let $C$ be a subset of $\HH$.
Then $\iota_C$ is the indicator function of $C$ and 
$\sigma_C=\iota_C^*$ is
the support function of $C$; if $C$ is nonempty, closed, and
convex, then $\proj_C=\prox_{\iota_C}$ is the projection operator
onto $C$. See \cite{Livre1} for background on Hilbertian convex
analysis and \cite{Rock70} for the Euclidean setting. 

\begin{definition}
\label{d:hautbas}
Let $f\colon\HH\to\RXX$. Then 
\begin{equation}
\label{e:66}
(\forall\xi\in\RP)\quad\xi\rocky f=
\begin{cases}
\iota_{\cdom f},&\text{if}\;\:\xi=0;\\
\xi f,&\text{if}\;\:\xi>0.
\end{cases}
\end{equation} 
In addition,
\begin{equation}
\label{e:pbas}
\pbas{f}\colon\HH\to\RX\colon x\mapsto
\begin{cases}
f(x),&\text{if}\;\:\minf<f(x)<0;\\
\pinf,&\text{otherwise}
\end{cases}
\end{equation}
and the $\blacktriangledown$ envelope of $f$ is
$\bas{f}=f^{{\tiny\mbox{$\vee$}}**}$.
Furthermore, 
\begin{equation}
\label{e:phaut}
\phaut{f}\colon\HH\to\RX\colon x\mapsto
\begin{cases}
f(x),&\text{if}\;\:0<f(x)<\pinf;\\
\pinf,&\text{otherwise}
\end{cases}
\end{equation}
and the $\blacktriangleup$ envelope of $f$ is
$\haut{f}=f^{{\tiny\mbox{$\wedge$}}**}$.
\end{definition}

Let us record a few facts.

\begin{lemma}{\rm\cite[Proposition~13.15]{Livre1}}
\label{l:fy}
Let $f\colon\HH\to\RX$ be proper, let $x\in\HH$, and let
$x^*\in\HH$. Then $f(x)+f^*(x^*)\geq\scal{x}{x^*}$.
\end{lemma}

\begin{lemma}{\rm\cite[Theorem~3E]{Rock66}}
\label{l:66}
Let $f\in\Gamma_0(\HH)$ and $\gamma\in\RP$. Then the following
hold:
\begin{enumerate}
\item
\label{l:66i}
$\gamma\rocky f\in\Gamma_0(\HH)$.
\item
\label{l:66ii}
$[\widetilde{f}(\cdot,\gamma)]^*=\gamma\rocky f^*$ and
$(\gamma\rocky f)^*=\widetilde{f^*}(\cdot,\gamma)$. 
\end{enumerate}
\end{lemma}

\begin{lemma}{\rm\cite[Lemma~3.2]{Vina23}}
\label{l:12}
Let $f\in\Gamma_0(\HH)$ be such that $f^{-1}(\RMM)\neq\emp$.
Then the following hold:
\begin{enumerate}
\item
\label{l:12i} 
$\bas{f}\in\Gamma_{0}(\HH)$.
\item
\label{l:12ii} 
$\dom\bas{f}=\overline{f^{-1}(\RMM)}=f^{-1}(\RM)$.
\item
\label{l:12iii} 
Let $x\in\HH$ be such that $f(x)\in\RM$. Then $\bas{f}(x)=f(x)$.
\end{enumerate}
\end{lemma}

\begin{lemma}
\label{l:24}
Let $f\in\Gamma_0(\HH)$ be such that $f^{-1}(\RPP)\neq\emp$. 
Then the following hold:
\begin{enumerate}
\item
\label{l:24i}
$\haut{f}\in\Gamma_0(\HH)$.
\item
\label{l:24ii}
$\dom\haut{f}=\dom f\cap\cconv f^{-1}(\RPP)$.
\item
\label{l:24iii}
$\haut{f}(\dom\haut{f})\subset\RP$.
\item
\label{l:24iv}
Let $x\in\HH$ be such that $f(x)\in\RPP$. Then $\haut{f}(x)=f(x)$.
\item
\label{l:24v}
$\cdom\haut{f}=\cconv f^{-1}(\RPP)$.
\end{enumerate}
\end{lemma}
\begin{proof}
\ref{l:24i}--\ref{l:24iv}: See \cite[Lemma~3.3]{Vina23}.

\ref{l:24v}: 
By \ref{l:24iv}, $f^{-1}(\RPP)\subset\dom\haut{f}$. Hence, 
since $\haut{f}$ is convex by \ref{l:24i}, 
$\dom\haut{f}$ is convex, and therefore
$\cconv f^{-1}(\RPP)\subset\cdom \haut{f}$.
The assertion therefore follows from \ref{l:24ii}.
\end{proof}

\section{Preliminary results}
\label{sec:3}

We establish results on which the derivations of
Section~\ref{sec:4} will rest.

\begin{lemma}
\label{l:j1}
Let $f\in\Gamma_0(\HH)$, $x\in\HH$, $p\in\HH$, and $\gamma\in\RP$.
Then the following hold:
\begin{enumerate}
\item
\label{l:j1-i}
$\prox_{\gamma\rocky f}=
\begin{cases}
\proj_{\cdom f}, &\text{if}\;\;\gamma=0;\\
\prox_{\gamma f}, &\text{if}\;\;\gamma\in\RPP.
\end{cases}$
\item
\label{l:j1-+i}
$\range\prox_{\gamma\rocky f}\subset
\dom(\gamma\rocky f)\subset\cdom f$.
\item
\label{l:j1iii}
$p=\prox_{\gamma\rocky f}x$ $\Leftrightarrow$ $(\forall y\in\HH)$ 
$\scal{y-p}{x-p}+(\gamma\rocky f)(p)\leq(\gamma\rocky f)(y)$.
\item
\label{l:j1i}
$p=\prox_{\gamma\rocky f}x$ $\Leftrightarrow$ 
$(\gamma\rocky f)(p)+(\gamma\rocky f)^{*}(x-p)=\scal{p}{x-p}$.
\item
\label{l:j1ii-}
Suppose that $\gamma>0$. Then
$p=\prox_{\gamma f}x$ $\Leftrightarrow$ 
$f(p)+f^{*}((x-p)/\gamma)=\scal{p}{x-p}/\gamma$.
\item
\label{l:j1ii}
Suppose that $\gamma>0$. Then
$x=\prox_{\gamma f}x+\gamma\,\prox_{\fetoile/\gamma}(x/\gamma)$.
\end{enumerate}
\end{lemma}
\begin{proof}
Recall from Lemma~\ref{l:66}\ref{l:66i} that 
$\gamma\rocky f\in\Gamma_0(\HH)$.

\ref{l:j1-i}: This follows from \eqref{e:66}.

\ref{l:j1-+i}: This follows from \eqref{e:p} 
and \eqref{e:66}.

\ref{l:j1iii}: In view of \eqref{e:66}, for $\gamma=0$, this is the
characterization of the projection of $x$ onto the nonempty closed
convex set $\cdom f$ \cite[Theorem~3.16]{Livre1} while, for
$\gamma>0$, this is \cite[Proposition~12.26]{Livre1}.

\ref{l:j1i}: 
By virtue of \ref{l:j1-+i},
$\dom(\gamma\rocky f)\subset\cdom f$.
Hence, Lemma~\ref{l:fy} and \eqref{e:c} yield
\begin{equation}
\label{e:o1}
\scal{p}{x-p}
\leq(\gamma\rocky f)(p)+(\gamma\rocky f)^*(x-p)
=\sup_{y\in\cdom f}\bigl(\scal{y}{x-p}
+(\gamma\rocky f)(p)-(\gamma\rocky f)(y)\bigr).
\end{equation}
On the other hand, we derive from \ref{l:j1iii} that
\begin{equation}
\label{e:o2}
p=\prox_{\gamma\rocky f}x\quad\Leftrightarrow\quad
\sup_{y\in\cdom f}\bigl(\scal{y}{x-p}
+(\gamma\rocky f)(p)-(\gamma\rocky f)(y)\bigr)
\leq\scal{p}{x-p}.
\end{equation}
Combining \eqref{e:o1} and \eqref{e:o2} furnishes the desired
characterization.

\ref{l:j1ii-}: This follows from \ref{l:j1i} and 
Lemma~\ref{l:66}\ref{l:66ii}.

\ref{l:j1ii}: See \cite[Proposition~14.3{\rm(ii)}]{Livre1}.
\end{proof}

\begin{lemma}
\label{l:5}
Let $\gamma\in\RP$, let $\phi\in\Gamma_0(\RR)$ be even and
such that $0\in\inte\dom\phi$, set
$\varphi=\phi\circ\|\cdot\|$, and let $x\in\HH$. Then
$\varphi\in\Gamma_0(\HH)$ and the following hold:
\begin{enumerate}
\item 
\label{l:5ii}
$\prox_{\gamma\rocky \varphi}x=
\begin{cases}
\dfrac{\prox_{\gamma\rocky\phi}\|x\|}{\|x\|}x,
&\text{if}\:\:x\neq 0;\\
0,&\text{if}\:\:x=0.
\end{cases}$
\item 
\label{l:5iii}
$\varphi(\prox_{\gamma\rocky \varphi}x)=
\phi(\prox_{\gamma\rocky\phi}\|x\|)$.
\end{enumerate}
\end{lemma}
\begin{proof}
Since \ref{l:5iii} follows from \ref{l:5ii}, we prove the latter.
We have
$\varphi\in\Gamma_0(\HH)$. In addition, by
\cite[Propositions~16.17(ii) and 16.27]{Livre1},
$\partial\phi(0)$ is a symmetric compact interval, say
$\partial\phi(0)=[-\tau,\tau]$, where $\tau\in\RP$.
We also note that there exists $\rho\in\RPPX$ 
such that
\begin{equation}
\label{e:cabron}
\cdom\phi=
\begin{cases}
[-\rho,\rho],&\text{if}\;\;\rho<\pinf;\\
\RR,&\text{if}\;\;\rho=\pinf
\end{cases}
\quad\text{and}\quad
\cdom\varphi=
\begin{cases}
B(0;\rho), &\text{if}\;\;\rho<\pinf;\\
\HH,&\text{if}\;\;\rho=\pinf.
\end{cases}
\end{equation}
If $\rho<+\infty$, we derive from \eqref{e:cabron} and 
\cite[Example~3.18]{Livre1} that
\begin{equation}
\label{e:nmtmabc}
\proj_{\cdom\varphi}x=
\dfrac{\rho x}{\max\{\|x\|,\rho\}}=
\begin{cases}
\dfrac{\rho\|x\|}{\max\{\|x\|,\rho\}\|x\|}x,
&\text{if}\:\:x\neq 0;\\
0,&\text{if}\:\:x=0
\end{cases}
=
\begin{cases}
\dfrac{\proj_{\cdom\phi}\|x\|}{\|x\|}x,
&\text{if}\:\:x\neq 0;\\
0,&\text{if}\:\:x=0
\end{cases}
\end{equation}
whereas, if $\rho=+\infty$, it is clear that 
$\proj_{\cdom\varphi}x$ coincides with the last term above.
In view of Lemma~\ref{l:j1}\ref{l:j1-i}, this establishes the
claim for $\gamma=0$. Now 
suppose that $\gamma>0$. Then it follows from 
\cite[Proposition~2.1]{Nmtm09} that 
\begin{equation}
\label{e:nmtma}
\prox_{\gamma\varphi}x=
\begin{cases}
\dfrac{\prox_{\gamma\phi}\|x\|}{\|x\|}x,
&\text{if}\:\:\|x\|>\gamma\tau;\\
0,&\text{if}\:\:\|x\|\leq\gamma\tau.
\end{cases}
\end{equation}
Moreover, since, in view of \eqref{e:p}, $\|x\|\leq\gamma\tau$ 
$\Leftrightarrow$ 
$\|x\|\in\gamma\partial\phi(0)$ 
$\Leftrightarrow$ 
$\prox_{\gamma\phi}\|x\|=0$, 
\eqref{e:nmtma} reduces to 
\begin{equation}
\label{e:nmtmab}
\prox_{\gamma\varphi}x=
\begin{cases}
\dfrac{\prox_{\gamma\phi}\|x\|}{\|x\|}x,
&\text{if}\:\:x\neq 0;\\
0,&\text{if}\:\:x=0,
\end{cases}
\end{equation}
as required.
\end{proof}

\begin{lemma}
\label{l:j2}
Let $f\in\Gamma_0(\HH)$, let $x\in\HH$, and set 
$\phi\colon\RP\to\RX\colon\gamma\mapsto f(\prox_{\gamma\odot f}x)$.
Then the following hold:
\begin{enumerate}
\item
\label{l:j2i-}
Let $\mu\in\RP$ and $\gamma\in\left]\mu,\pinf\right[$. Then
$\phi(\gamma)\leq\phi(\mu)-
\|\prox_{\mu\rocky f}x-\prox_{\gamma f}x\|^2/(\gamma-\mu)$.
\item
\label{l:j2i}
$\phi$ is decreasing on $\RP$.
\item
\label{l:j2iii}
$\phi$ is continuous.
\end{enumerate}
\end{lemma}
\begin{proof}
First note that Lemma~\ref{l:66}\ref{l:66i} guarantees that 
$\prox_{\gamma\rocky f}$ and, therefore $\phi$, are
well defined.

\ref{l:j2i-}: 
Set $p=\prox_{\mu\rocky f}x$ and $q=\prox_{\gamma f}x$, and note
that \eqref{e:p} implies that $q\in\dom f$. If $\mu=0$, we assume
that $p=\proj_{\cdom f}x\in\dom f$ since, otherwise,
$\phi(\mu)=\pinf$ and the inequality holds trivially. By
Lemma~\ref{l:j1}\ref{l:j1iii}, $\scal{q-p}{x-p}\leq\mu(f(q)-f(p))$
and $\scal{p-q}{x-q}\leq\gamma(f(p)-f(q))$. Adding these
inequalities yields
\begin{equation}
\|p-q\|^2\leq(\gamma-\mu)\bigl(f(p)-f(q)\bigr)=
(\gamma-\mu)\bigl(\phi(\mu)-\phi(\gamma)\bigr), 
\end{equation}
which is equivalent to the announced inequality.

\ref{l:j2i}: Clear from \ref{l:j2i-}.

\ref{l:j2iii}: 
Set $T\colon\RP\to\HH\colon\gamma\mapsto\prox_{\gamma\rocky f}x$.
It follows from \cite[Proposition~23.31(iii)]{Livre1} applied to
the maximally monotone operator $\partial f$ that $T$ is continuous
on $\RPP$ and from \cite[Theorem~23.48]{Livre1} that it is
right-continuous at $0$. Now suppose that 
$(\gamma_n)_{n\in\NN}$ be a sequence in $\RPP$ such that 
$\gamma_n\to\mu\in\RP$.
Then $T(\gamma_n)\to T(\mu)$. If $\mu=0$, by invoking the
lower semicontinuity of $f$ and \ref{l:j2i}, we get
\begin{equation}
\label{e:12}
\phi(0)=f\bigl(T(\mu)\bigr)\leq\varliminf f\bigl(T(\gamma_n)\bigr)
=\varliminf\phi(\gamma_n)
\leq\varlimsup\phi(\gamma_n)\leq\phi(0)
\end{equation}
and therefore $\phi(\gamma_n)\to\phi(0)$.
If $\mu>0$, the continuity of $\phi$ at $\mu$ is established in 
\cite[Lemma~3.27]{Atto84}.
\end{proof}

The following proposition provides explicit expressions for the
perspective function of Definition~\ref{d:1} as well as conditions
that guarantee that the perspective $\varphi\persp s$ is in
$\Gamma_0(\HH\oplus\GG)$, and hence that $\prox_{\varphi\persp s}$
is well defined.

\begin{proposition}
\label{p:2}
Let $\varphi\in\Gamma_0(\HH)$ and let
$s\colon\GG\to\RXX$ be such that $S=s^{-1}(\RPP)\neq\emp$.
Let $x\in\HH$ and $y\in\GG$. Then the following hold: 
\begin{enumerate}
\item
\label{p:2iii} 
Suppose that $\varphi^*(\HH)\subset\RPX$,
$(\varphi^*)^{-1}(\RPP)\neq\emp$, and $-s\in\Gamma_0(\GG)$. Then 
\begin{equation}
\label{e:975}
(\varphi\persp s)(x,y)=
\begin{cases}
s(y){\varphi}\biggl(\dfrac{x}{s(y)}\biggr),
&\text{if}\;\;0<s(y)<\pinf;\\
(\rec\varphi)(x),&\text{if}\;\;s(y)=0;\\
\pinf,&\text{otherwise}.
\end{cases}
\end{equation}
\item
\label{p:2ii} 
Suppose that $\varphi^*(\HH)\subset\{0,\pinf\}$. Then
$(\varphi\persp s)(x,y)=\varphi(x)+\iota_{\cconv S}(y)$.
\item
\label{p:2i} 
Suppose that 
$\varphi^*(\HH)\subset\RM\cup\{\pinf\}$, 
$(\varphi^*)^{-1}(\RMM)\neq\emp$, and $s\in\Gamma_0(\GG)$. Then
\begin{equation} 
\label{e:365}
(\varphi\persp s)(x,y)=
\begin{cases}
s(y){\varphi}
\biggl(\dfrac{x}{{s}(y)}\biggr),
&\text{if}\;\;0<s(y)<\pinf;\\
(\rec\varphi)(x),&\text{if}\;\;y\in\cconv 
S\:\:\text{and}\:\:s(y)\leq 0;\\
\pinf,&\text{otherwise}.
\end{cases}
\end{equation} 
\end{enumerate}
Additionally, in each case, 
$\varphi\persp s\in\Gamma_0(\HH\oplus\GG)$. 
\end{proposition}
\begin{proof}
\ref{p:2iii}:
Since $\varphi\in\Gamma_0(\HH)$, \eqref{e:c} yields
\begin{equation}
\varphi(0)=\varphi^{**}(0)=\sup_{x^*\in\HH}-\varphi^*(x^*)
=-\inf_{x^*\in\HH}\varphi^*(x^*).
\end{equation}
Hence, 
\begin{equation}
\varphi(0)\leq 0\quad\Leftrightarrow\quad \varphi^*(\dom 
\varphi^*)\subset\RP.
\end{equation}
The result therefore follows from \cite[Lemma~2.5(iii)]{Vina23} and 
\cite[Corollary~5.3(iii)]{Vina23}.

\ref{p:2ii}: This follows from \cite[Lemma~2.5(iii)]{Vina23} and 
\cite[Corollary~5.3(ii)]{Vina23}.

\ref{p:2i}:
Since $\varphi\in\Gamma_0(\HH)$, 
\cite[Lemma~2.6]{Vina23} yields
\begin{equation}
\rec\varphi\leq \varphi\quad\Leftrightarrow\quad \varphi^*(\dom 
\varphi^*)\subset\RM.
\end{equation}
Hence, the result follows from \cite[Lemma~2.5(iii)]{Vina23} and 
\cite[Corollary~5.3(i)]{Vina23}.
\end{proof}

\section{Computation of the proximity operator}
\label{sec:4}
Our strategy to compute the proximity operator of perspective
functions is to use the following results based on the
Fenchel--Young identity.

\begin{proposition}
\label{p:8}
Suppose that $\varphi\in\Gamma_0(\HH)$ and $s\colon\GG\to\RXX$
satisfy the conditions of Proposition~\ref{p:2}. Let $x\in\HH$, 
$y\in\GG$, and $\gamma\in\RPP$. Set 
$(p,q)=\prox_{\gamma (\varphi\persp s)}(x,y)$. Then 
the following hold: 
\begin{enumerate}
\item
\label{p:8i}
$(\varphi\persp s)(p,q)+(\varphi\persp s)^*
\biggl(\dfrac{x-p}{\gamma},\dfrac{y-q}{\gamma}\biggr)
=\Scal{p}{\dfrac{x-p}{\gamma}}+\Scal{q}{\dfrac{y-q}{\gamma}}$.
\item
\label{p:8ii}
$(p,q)\in\dom(\varphi\persp s)$.
\item
\label{p:8iii}
$\gamma^{-1}(x-p,y-q)\in\dom(\varphi\persp s)^*$.
\end{enumerate}
\end{proposition}
\begin{proof}
We note that $(p,q)$ is well defined by virtue of 
Proposition~\ref{p:2}.

\ref{p:8i}: This follows from Lemma~\ref{l:j1}\ref{l:j1ii-}.

\ref{p:8ii}--\ref{p:8iii}: These follow from \ref{p:8i}.
\end{proof}

To implement the above strategy, explicit expressions are required
for $(\varphi\persp s)^*$.

\begin{proposition}{\rm\cite[Theorem~4.5]{Vina23}}
\label{p:3}
Let $\varphi\colon\HH\to\RX$ be proper, let $s\colon\GG\to\RXX$
be such that $S=s^{-1}(\RPP)\neq\emp$, let $x^*\in\HH$, and
let $y^*\in\GG$. Then the following hold: 
\begin{enumerate}
\item 
\label{p:3i} 
Suppose that $\varphi^*(\HH)\subset\RPX$ and 
$(\varphi^*)^{-1}(\RPP)\neq\emp$.
Then
\begin{equation}
\label{e:r7}
(\varphi\persp s)^*(x^*,y^*)=
\begin{cases}
\varphi^*(x^*)\:\base{(-s)}\Bigg(\dfrac{y^*}{\varphi^*(x^*)}\Bigg),
&\text{if}\:\:0<\varphi^*(x^*)<\pinf;\\
\sigma_{\cconv S}(y^*),&\text{if}\:\:\varphi^*(x^*)=0;\\[3mm]
\pinf,&\text{if}\:\:\varphi^*(x^*)=\pinf.
\end{cases}
\end{equation}
\item 
\label{p:3ii} 
Suppose that $\varphi^*(\HH)\subset\{0,\pinf\}$. Then 
\begin{equation} 
\label{e:38conj}
\bigl(\varphi
\persp{s}\bigr)^{*}(x^*,y^*)=
\iota_{(\varphi^*)^{-1}(\{0\})}(x^*)+\sigma_{\cconv S}(y^*).
\end{equation} 
\item 
\label{p:3iii} 
Suppose that $\varphi^*(\HH)\subset\RM\cup\{\pinf\}$ and 
$(\varphi^*)^{-1}(\RMM)\neq\emp$. Then 
\begin{equation} 
\label{e:37conj}
\bigl(\varphi\persp{s}\bigr)^{*}(x^*,y^*)=
\begin{cases}
-\varphi^*(x^*)\:\haute{s}\Bigg(\dfrac{y^*}{-\varphi^*(x^*)}
\Bigg),&\text{if}\:\:\minf<\varphi^*(x^*)<0;\\
\sigma_{\cconv S}(y^*),&\text{if}\:\:\varphi^*(x^*)=0;\\[3mm]
\pinf,&\text{if}\:\:\varphi^*(x^*)=\pinf.
\end{cases}
\end{equation} 
\end{enumerate}
\end{proposition}

We are now ready to present our main result.

\begin{theorem}
\label{t:2}
Let $\varphi\in\Gamma_0(\HH)$ and let $s\colon\GG\to\RXX$ be such
that $S=s^{-1}(\RPP)\neq\emp$. Let $x\in\HH$, $y\in\GG$, and 
$\gamma\in\RPP$. Then the following hold: 
\begin{enumerate}
\item
\label{t:2i} 
Suppose that $\varphi^*(\HH)\subset\RPX$,
$(\varphi^*)^{-1}(\RPP)\neq\emp$, and $-s\in\Gamma_0(\GG)$. 
Then there exists a unique $\eta\in\RP$ such that 
\begin{equation}
\label{e:mainr}
\bas{(-s)}\biggl(\prox_{\gamma
\varphie\bigl(\prox_{\frac{\eta}{\gamma}\rocky\varphie}
\bigl(\frac{x}{\gamma}\bigr)\bigr)\rocky\bas{(-s)}}\,y\biggr)
+\eta=0.
\end{equation}
Furthermore,
\begin{equation}
\label{e:Kj}
\prox_{\gamma(\varphi\persp s)}(x,y)
=\biggl(x-\gamma\,
\prox_{\frac{\eta}{\gamma}\rocky\varphie}
\Big(\frac{x}{\gamma}\Big),\prox_{\gamma
\varphie\bigl(\prox_{\frac{\eta}{\gamma}\rocky\varphie}
\bigl(\frac{x}{\gamma}\bigr)\bigr)\rocky\bas{(-s)}}\, y\biggr).
\end{equation}
\item
\label{t:2ii} 
Suppose that $\varphi^*(\HH)\subset\{0,\pinf\}$. Then 
$\prox_{\gamma(\varphi\persp s)}(x,y)=
(\prox_{\gamma\varphi}\,x,\proj_{\cconv S}\,y)$.
\item
\label{t:2iii} 
Suppose that 
$\varphi^*(\HH)\subset\RM\cup\{\pinf\}$, 
$(\varphi^*)^{-1}(\RMM)\neq\emp$, and $s\in\Gamma_0(\GG)$. 
Then there exists a unique $\eta\in\RP$ such that 
\begin{equation}
\label{e:mainriii}
\varphi^*\biggl(\prox_{\frac{1}{\gamma}
\haut{s}(\prox_{\gamma\eta\rocky\haut{s}} y)\rocky\varphie}\,
\Big(\frac{x}{\gamma}\Big)\biggr)+\eta=0.
\end{equation}
Furthermore,
\begin{equation}
\label{e:Kjiii}
\prox_{\gamma(\varphi\persp s)}(x,y)
=\biggl(x-\gamma\,\prox_{\frac{1}{\gamma}\haut{s}
(\prox_{\gamma\eta\rocky\haut{s}}y)\rocky\varphie}
\Big(\frac{x}{\gamma}\Big),
\prox_{\gamma\eta\rocky\haut{s}}y\biggr).
\end{equation}
\end{enumerate}
\end{theorem}
\begin{proof}
Set $(p,q)=\prox_{\gamma (\varphi\persp s)}(x,y)$.

\ref{t:2i}:
We deduce from Proposition~\ref{p:8}\ref{p:8iii} and 
Proposition~\ref{p:3}\ref{p:3i} that
\begin{equation}
\label{e:xp}
\varphi^*\biggl(\dfrac{x-p}{\gamma}\biggr)\in\RP
\end{equation}
and from Proposition~\ref{p:8}\ref{p:8ii} and
Proposition~\ref{p:2}\ref{p:2iii} that 
\begin{equation}
\label{e:q}
s(q)\in\RP.
\end{equation}
Since $\varphi\in\Gamma_0(\HH)$, we have $\varphi^{**}=\varphi$
\cite[Corollary~13.38]{Livre1}. Hence, it follows from 
Proposition~\ref{p:2}\ref{p:2iii}, \eqref{e:0}, 
Lemma~\ref{l:66}\ref{l:66ii},
and \eqref{e:q} that
\begin{equation}
\label{e:a1}
(\varphi\persp s)(p,q)=
\widetilde{\varphi}\bigl(p,s(q)\bigr)
=\bigl(s(q)\rocky\varphi^*\bigr)^*(p).
\end{equation}
Next, since Lemma~\ref{l:12}\ref{l:12i} asserts that
$\bas{(-s)}\in\Gamma_0(\GG)$, we have
$\base{(-s)}\in\Gamma_0(\GG)$ and hence deduce from 
Lemma~\ref{l:12}\ref{l:12ii} and \cite[Proposition~13.49]{Livre1} 
that
\begin{equation}
\label{e:s}
\sigma_{\cconv S}=\sigma_{\overline{S}}=\sigma_{\dom\bas{(-s)}}
=\rec{\base{(-s)}}.
\end{equation}
Thus, it follows from \eqref{e:xp}, 
Proposition~\ref{p:3}\ref{p:3i}, \eqref{e:0}, 
and Lemma~\ref{l:66}\ref{l:66ii} that 
\begin{align}
&\hskip -6mm
(\varphi\persp s)^*\biggl(\dfrac{x-p}{\gamma},
\dfrac{y-q}{\gamma}\biggr)\nonumber\\
&=
\begin{cases}
\varphi^*\biggl(\dfrac{x-p}{\gamma}\biggr)
\base{(-s)}\biggl(\dfrac{(y-q)/\gamma}
{\varphi^*\bigl((x-p)/\gamma\bigr)}\biggr),
&\text{if}\;\;0<\varphi^*\biggl(\dfrac{x-p}{\gamma}\biggr)
<\pinf;\\[4mm]
\Big(\rec{\base{(-s)}}\Big)\biggl(\dfrac{y-q}{\gamma}\biggr),
&\text{if}\;\;\varphi^*\biggl(\dfrac{x-p}{\gamma}\biggr)=0
\end{cases}
\label{e:a7}\\[2mm]
&=\widetilde{\base{(-s)}}\biggl(\dfrac{y-q}{\gamma},
\varphi^*\biggl(\dfrac{x-p}{\gamma}\biggr)\biggr)
\nonumber\\[3mm]
&=\biggl(\varphi^*\biggl(\dfrac{x-p}{\gamma}\biggr)
\rocky\bas{(-s)}\biggr)^*\biggl(\dfrac{y-q}{\gamma}\biggr).
\label{e:a2}
\end{align}
On the other hand, \eqref{e:xp} and \eqref{e:q} yield 
$(x-p)/\gamma\in(\varphi^*)^{-1}(\RP)$ and $q\in\dom s$, 
respectively. Therefore, since
\eqref{e:q} and Lemma~\ref{l:12}\ref{l:12iii} yield
\begin{equation}
\label{e:sq}
0\leq s(q)=-\bas{(-s)}(q),
\end{equation}
we deduce from \eqref{e:66} that
\begin{equation}
\label{e:g4}
\bigl(s(q)\rocky\varphi^*\bigr)\biggl(\dfrac{x-p}{\gamma}\biggr)
+\biggl(\varphi^*\biggl(\dfrac{x-p}{\gamma}\biggr)
\rocky\bas{(-s)}\biggr)(q)
=0.
\end{equation}
Consequently, it results from \eqref{e:a1}, \eqref{e:a2}, and
Proposition~\ref{p:8}\ref{p:8i} that 
\begin{align}
\label{e:p1}
&\hspace{-.9cm}\bigl(s(q)\rocky\varphi^*\bigr)^*(p)
+\bigl(s(q)\rocky\varphi^*\bigr)\biggl(\dfrac{x-p}{\gamma}\biggr)
+\biggl(\varphi^*\biggl(\dfrac{x-p}{\gamma}\biggr)
\rocky\bas{(-s)}\biggr)(q)
+\biggl(\varphi^*\biggl(\dfrac{x-p}{\gamma}\biggr)\rocky
\bas{(-s)}\biggr)^*\biggl(\dfrac{y-q}{\gamma}\biggr)
\nonumber\\
&\hspace{5cm}=
\bigl(s(q)\rocky\varphi^*\bigr)^*(p)+
\biggl(\varphi^*\biggl(\dfrac{x-p}{\gamma}\biggr)\rocky
\bas{(-s)}\biggr)^*\biggl(\dfrac{y-q}{\gamma}\biggr)\nonumber\\
&\hspace{5cm}=(\varphi\persp s)(p,q)+(\varphi\persp s)^*
\biggl(\dfrac{x-p}{\gamma},\dfrac{y-q}{\gamma}\biggr)\nonumber\\
&\hspace{5cm}=\Scal{p}{\frac{x-p}{\gamma}}
+\Scal{q}{\frac{y-q}{\gamma}}.
\end{align}
We therefore derive from Lemma~\ref{l:fy} and 
Lemma~\ref{l:66}\ref{l:66i} that 
\begin{equation}
\label{e:x1}
\bigl(s(q)\rocky\varphi^*\bigr)^*(p)+
\bigl(s(q)\rocky\varphi^*\bigr)^{**}\biggl(\dfrac{x-p}{\gamma}\biggr)
=\Scal{p}{\dfrac{x-p}{\gamma}}
\end{equation}
and
\begin{equation}
\label{e:x2}
\biggl(\varphi^*\biggl(\dfrac{x-p}{\gamma}\biggr)
\rocky\bas{(-s)}\biggr)(q)+
\biggl(\varphi^*\biggl(\dfrac{x-p}{\gamma}\biggr)\rocky\bas{(-s)}
\biggr)^*
\biggl(\dfrac{y-q}{\gamma}\biggr)
=\Scal{q}{\dfrac{y-q}{\gamma}}.
\end{equation}
In turn, \eqref{e:x1} and 
Lemma~\ref{l:j1}\ref{l:j1ii-}--\ref{l:j1ii} yield
\begin{equation}
\label{e:23}
p=\prox_{\gamma (s(q)\rocky\varphie)^{\setoile}}x=x-\gamma\,
\prox_{\frac{s(q)}{\gamma}\rocky\varphie}
\biggl(\dfrac{x}{\gamma}\biggr),
\end{equation}
while \eqref{e:x2} and Lemma~\ref{l:j1}\ref{l:j1ii-} yield
\begin{equation}
\label{e:24}
q=\prox_{\gamma\varphie((x-p)/\gamma)\rocky\bas{(-s)}}\,y.
\end{equation}
Upon combining \eqref{e:23} and \eqref{e:24}, we obtain
\begin{equation}
\label{e:25}
q=\prox_{\gamma\varphie
\bigl(\prox_{\frac{s(q)}{\gamma}\rocky\varphie}
\bigl(\frac{x}{\gamma}\bigr)\bigr)\rocky\bas{(-s)}}\,y.
\end{equation}
Consequently, we deduce from \eqref{e:sq}
that $\eta=s(q)\in\RP$ solves \eqref{e:mainr}, from which 
\eqref{e:Kj} follows. To establish the uniqueness of the solution
to \eqref{e:mainr}, define
\begin{equation}
\begin{cases}
\label{e:r-}
\phi_1\colon\RP\to\RPX\colon\eta\mapsto
\varphi^*\Big(\prox_{\frac{\eta}{\gamma}\rocky\varphie}
\bigl(\frac{x}{\gamma}\bigr)\Big)\\[4mm]
\phi_2\colon\RPX\to\RMX\colon\mu\mapsto
\begin{cases}
\bas{(-s)}\bigl(\prox_{\gamma\mu\rocky\bas{(-s)}}\,y\bigr),
&\text{if}\:\:\mu<\pinf;\\
\inf\bas{(-s)},&\text{if}\:\:\mu=\pinf,
\end{cases}
\end{cases}
\end{equation}
and note that $\phi_2$ is well defined, since 
Lemma~\ref{l:j1}\ref{l:j1-+i} and Lemma~\ref{l:12}\ref{l:12ii}
yield, for every $\mu\in\RP$,
$\prox_{\gamma 
\mu\rocky\bas{(-s)}}y\in\cdom\bas{(-s)}=\dom\bas{(-s)}$.
Since \eqref{e:mainr} is equivalent to $\psi(\eta)=0$, 
where $\psi=\phi_2\circ\phi_1+\Id\colon\RP\to\RM$, let us prove
that $\psi$ is continuous and strictly increasing in $\RP$. Indeed, 
if $\proj_{\cdom\varphi^*}(x/\gamma)\in\dom\varphi^*$, 
Lemma~\ref{l:j1}\ref{l:j1-i} yields $\phi_1(\RP)\subset\RP$ and
Lemma~\ref{l:j2}\ref{l:j2i}--\ref{l:j2iii} implies that $\psi$ is 
strictly increasing and continuous. On the other hand, if 
$\proj_{\cdom\varphi^*}(x/\gamma)\notin\dom \varphi^*$, it follows 
from Lemma~\ref{l:j2}\ref{l:j2iii} that 
\begin{equation}
\phi_1(\eta)=\varphie\biggl(\prox_{\frac{\eta}{\gamma}\rocky\varphie}
\Big(\frac{x}{\gamma}\Big)\biggr)\uparrow\phi_1(0)
=\pinf\quad\text{as}\quad 
\eta\downarrow0
\end{equation}
and Lemma~\ref{l:j2}\ref{l:j2i} and
\cite[Proposition~12.33(i)]{Livre1} yield
\begin{align}
\label{e:late}
\phi_2\bigl(\phi_1(\eta)\bigr)&=\bas{(-s)}\biggl(\prox_{\gamma
\varphie\bigl(\prox_{\frac{\eta}{\gamma}\rocky\varphie}
\bigl(\frac{x}{\gamma}\bigr)\bigr)\rocky\bas{(-s)}}\,y\biggr)\downarrow
\inf\bas{(-s)}=\phi_2\bigl(\phi_1(0)\bigr)\quad\text{as}\quad 
\eta\downarrow0. 
\end{align}
Moreover, since $S\neq\emp$, we have $\inf\bas{(-s)}\in\RMMX$ and 
deduce from \eqref{e:late} that any solution to \eqref{e:mainr} is 
strictly positive.
Altogether, \eqref{e:mainr} has at most one solution in $\RP$.

\ref{t:2ii}: 
This follows from Proposition~\ref{p:2}\ref{p:2ii}
and \cite[Proposition~24.11]{Livre1}.

\ref{t:2iii}: 
Lemma~\ref{l:j1}\ref{l:j1ii} asserts that
\begin{equation}
\label{e:34b}
(p,q)=(x,y)-\gamma\,\prox_{(\varphi\persp s)^{\setoile}/\gamma}
\biggl(\dfrac{x}{\gamma},\dfrac{y}{\gamma}\biggr). 
\end{equation}
On the other hand, Lemma~\ref{l:24}\ref{l:24v} yields
$\cconv S=\cdom\haut{s}$. It therefore follows from 
Lemma~\ref{l:24}\ref{l:24i} and 
\cite[Proposition~13.49]{Livre1}
that
\begin{equation}
\label{e:ss}
\rec\bigl({\haute{s}}\bigr)=\sigma_{\dom\haut{s}}
=\sigma_{\cdom\haut{s}}=\sigma_{\cconv S}.
\end{equation}
Moreover, items \ref{l:24i}, \ref{l:24iii}, and \ref{l:24iv} in 
Lemma~\ref{l:24} yield
\begin{equation}
\label{e:iiia1}
0\leq\haut{s}\in\Gamma_0(\GG)\quad\text{and}\quad 
(\haut{s})^{-1}(\RPP)=S\neq\emp,
\end{equation}
while Lemma~\ref{l:12}\ref{l:12iii} yields
\begin{equation}
\label{e:iiia2}
\bas{\varphi^*}=\varphi^*\in\Gamma_0(\HH)\quad\text{and}\quad 
(\varphi^*)^{-1}(\RMM)\neq\emp.
\end{equation}
In turn, by virtue of 
Proposition~\ref{p:3}\ref{p:3iii}, \eqref{e:ss}, and 
Proposition~\ref{p:2}\ref{p:2iii}, we obtain
\begin{equation}
\label{e:n}
(\varphi\persp s)^*\colon(x^*,y^*)
\mapsto\bigl(\haute{s}\persp(-\varphi^*)\bigr)(y^*,x^*). 
\end{equation}
Now set
$(r,t)=\prox_{(\varphi\persp 
s)^{\setoile}/\gamma}(x/\gamma,y/\gamma)$. 
Then we derive from \eqref{e:n} and
\cite[Proposition~24.8(iv)]{Livre1} that 
\begin{equation}
(t,r)=\prox_{(\haute{s}\persp(-\varphi^{\setoile}))/\gamma}
\biggl(\dfrac{y}{\gamma},\dfrac{x}{\gamma}\biggr). 
\end{equation}
Therefore, \eqref{e:iiia1}, \eqref{e:iiia2}, and
\ref{t:2i} imply that 
\begin{equation}
(t,r)=\biggl(\frac{y}{\gamma}-\frac{1}{\gamma}
\prox_{\gamma\eta\rocky\haut{s}}y,
\prox_{\frac{1}{\gamma}\haut{s}
(\prox_{\gamma\eta\rocky\haut{s}}y)\rocky\varphie}
\biggl(\frac{x}{\gamma}\biggr)\biggr),
\end{equation}
where $\eta$ is the unique solution in $\RP$ to
\eqref{e:mainriii}. The conclusion then comes from \eqref{e:34b}.
\end{proof}

Next, we provide explicit formulas for 
$\prox_{\gamma(\varphi\persp s)}(x,y)$ in items \ref{t:2i} and 
\ref{t:2iii} of
Theorem~\ref{t:2} (item \ref{t:2ii} is already explicit).

\begin{proposition}
\label{p:1}
Consider the assumptions and notation of 
Theorem~\ref{t:2}\ref{t:2i}, and set
\begin{equation}
\label{e:Omega_i}
\hskip -3.0mm
\begin{cases}
\Omega_1=\mEnge{(u,v)\in\HH\times\GG}
{\varphi^*\Big(\proj_{\cdom\varphie}
\Big(\dfrac{u}{\gamma}\Big)\Big)=0\;\;\text{and}\;\;
s(\proj_{\overline{S}}v)=0}\\[4mm]
\Omega_2=\mEnge{(u,v)\in\HH\times\GG}
{\varphi^*\Big(\proj_{\cdom\varphie}
\Big(\dfrac{u}{\gamma}\Big)\Big)\in\RPP
\;\;\text{and}\;
s\Big(\prox_{\gamma\varphie\bigl(\proj_{\cdom\varphie}
\bigl(\frac{u}{\gamma}\bigr)\bigr)\bas{(-s)}}\,v\Big)=0}\\[4mm]
\Omega_3=\mEnge{(u,v)\in\HH\times\GG}
{\varphi^*\Big(\prox_{\frac{s(\proj_{\overline{S}}v)}{\gamma}
\varphie}\Big(\dfrac{u}{\gamma}\Big)\Big)=0
\;\;\text{and}\;\;s(\proj_{\overline{S}}v)\in\RPP}\\[4mm]
\Omega_4=(\HH\times\GG)\smallsetminus
(\Omega_1\cup\Omega_2\cup\Omega_3).
\end{cases}
\end{equation}
Then exactly one of the following holds:
\begin{enumerate}
\item 
\label{p:1i}
$(x,y)\in\Omega_1$, 
$\eta=0$, and $\prox_{\gamma(\varphi\persp s)}(x,y)
=\Big(x-\gamma\,\proj_{\cdom\varphie}
(x/{\gamma}),\proj_{\overline{S}}y\Big)$.
\item 
\label{p:1iii}
$(x,y)\in\Omega_2$, $\eta=0$, and 
$\prox_{\gamma(\varphi\persp s)}(x,y)
=\Big(x-\gamma\,\proj_{\cdom\varphie}
(x/{\gamma}),
\prox_{\gamma\varphie\bigl(\proj_{\cdom\varphie}
(x/{\gamma})\bigr)\bas{(-s)}}\, y\Big)$.
\item 
\label{p:1ii}
$(x,y)\in\Omega_3$, 
$\eta=s(\proj_{\overline{S}}y)>0$, and 
$\prox_{\gamma(\varphi\persp s)}(x,y)=\Big(x-\gamma\,
\prox_{\frac{s(\proj_{\overline{S}}y)}{\gamma}\varphie}
(x/{\gamma}),\proj_{\overline{S}}y\Big)$.
\item 
\label{p:1iv}
$(x,y)\in\Omega_4$, $\eta>0$ solves 
\begin{equation}
\label{e:c6}
\eta=s\biggl(\prox_{\gamma
\varphie\bigl(\prox_{\frac{\eta}{\gamma}\varphie}
\bigl(\frac{x}{\gamma}\bigr)\bigr)\bas{(-s)}}\,y\biggr),
\end{equation}
and 
\begin{equation}
\label{e:c7}
\prox_{\gamma(\varphi\persp s)}(x,y)
=\biggl(x-\gamma\,
\prox_{\frac{\eta}{\gamma}\varphie}
\Big(\frac{x}{\gamma}\Big),\prox_{\gamma
\varphie\bigl(\prox_{\frac{\eta}{\gamma}\varphie}
\bigl(\frac{x}{\gamma}\bigr)\bigr)\bas{(-s)}}\, y\biggr).
\end{equation}
\end{enumerate}
\end{proposition}
\begin{proof} 
Lemma~\ref{l:12}\ref{l:12ii} yields
\begin{equation}
\label{e:7}
\dom\bas{(-s)}=\overline{(-s)^{-1}(\RMM)}=\overline{S}=s^{-1}(\RP).
\end{equation}
Hence, it follows from Lemma~\ref{l:j1}\ref{l:j1-+i} that
\begin{equation}
(\forall\mu\in\RP)\quad 
\prox_{\mu\rocky\bas{(-s)}}\,y\in s^{-1}(\RP).
\end{equation}
Therefore, Lemma~\ref{l:12}\ref{l:12iii} implies that
\eqref{e:mainr} is equivalent to 
\begin{equation}
\label{e:auxs}
\eta=s\biggl(\prox_{\gamma
\varphie\bigl(\prox_{\frac{\eta}{\gamma}\rocky\varphie}
\bigl(\frac{x}{\gamma}\bigr)\bigr)\rocky\bas{(-s)}}\,y\biggr).
\end{equation}

\ref{p:1i}: Since Lemma~\ref{l:j1}\ref{l:j1-i} and \eqref{e:7}
yield
\begin{equation}
s\biggl(\prox_{\gamma
\varphie\bigl(\proj_{\cdom\varphie}\bigl(\frac{x}{\gamma}\bigr)\bigr)
\rocky\bas{(-s)}}\,y\biggr)=s\bigl(\proj_{\overline{S}}y\bigr)=0,
\end{equation}
we deduce from \eqref{e:auxs} and Lemma~\ref{l:j1}\ref{l:j1-i} that
$\eta=0$. The claim therefore follows from \eqref{e:Kj}.

\ref{p:1iii}: Lemma~\ref{l:j1}\ref{l:j1-i} yields
\begin{equation}
s\biggl(\prox_{\gamma
\varphie\bigl(\proj_{\cdom\varphie}
\bigl(\frac{x}{\gamma}\bigr)\bigr)\rocky\bas{(-s)}}\,y\biggr)
=s\biggl(\prox_{\gamma\varphie\bigl(\proj_{\cdom\varphie}
\bigl(\frac{x}{\gamma}\bigr)\bigr)\bas{(-s)}}\,y\biggr)
=0,
\end{equation}
and we deduce from \eqref{e:auxs} and Lemma~\ref{l:j1}\ref{l:j1-i}
that $\eta=0$. Therefore, the claim follows from \eqref{e:Kj}.

\ref{p:1ii}: Since $s(\proj_{\overline{S}}y)>0$, 
Lemma~\ref{l:j1}\ref{l:j1-i} and \eqref{e:7} yield
\begin{align}
s\biggl(\prox_{\gamma
\varphie\Big(\prox_{\frac{s(\proj_{\overline{S}}y)}{\gamma}
\rocky\varphie}\bigl(\frac{x}{\gamma}\bigr)\Big)
\rocky\bas{(-s)}}\,y\biggr)&=s\biggl(\prox_{\gamma
\varphie\Big(\prox_{\frac{s(\proj_{\overline{S}}y)}{\gamma}
\varphie}\bigl(\frac{x}{\gamma}\bigr)\Big)\rocky\bas{(-s)}}\,y\biggr)
\nonumber\\
&=s\bigl(\proj_{\overline{S}}y\bigr),
\end{align}
and we deduce from \eqref{e:auxs} that 
$\eta=s(\proj_{\overline{S}}y)>0$. Therefore, the claim follows
from \eqref{e:Kj}.

\ref{p:1iv}: Suppose that $\eta=0$, hence
$\varphi^*(\proj_{\cdom\varphie}(x/\gamma))<\pinf$. 
Then it follows from \eqref{e:auxs} that 
\begin{equation}
\label{e:aux23}
0=s\biggl(\prox_{\gamma
\varphie\bigl(\proj_{\cdom\varphie}
\bigl(\frac{x}{\gamma}\bigr)\bigr)\rocky\bas{(-s)}}\,y\biggr).
\end{equation}
Therefore, if $\varphi^*(\proj_{\cdom\varphie}(x/\gamma))=0$,
then \eqref{e:aux23} yields $0=s(\proj_{\overline{S}}y)$, which
implies that $(x,y)\in\Omega_1$. On the other hand, if
$\varphi^*(\proj_{\cdom\varphie}(x/\gamma))\in\RPP$, then
\eqref{e:aux23} yields
\begin{equation}
0=s\Big(\prox_{\gamma\varphie\bigl(\proj_{\cdom\varphie}
\bigl(\frac{x}{\gamma}\bigr)\bigr)\bas{(-s)}}\,y\Big)
\end{equation}
and thus $(x,y)\in\Omega_2$. 
However, since $(x,y)\in\Omega_4$, we have
$(x,y)\notin\Omega_1\cup\Omega_2$ and obtain a contradiction.
This shows that $\eta>0$. In turn, \eqref{e:auxs} reduces to 
\begin{equation}
\label{e:auxs2}
\eta=s\biggl(\prox_{\gamma
\varphie\bigl(\prox_{\frac{\eta}{\gamma}\varphie}
\bigl(\frac{x}{\gamma}\bigr)\bigr)\rocky\bas{(-s)}}\,y\biggr).
\end{equation}
Hence, if 
$\varphi^*(\prox_{\frac{\eta}{\gamma}\varphie}(x/\gamma))=0$, 
we deduce from \eqref{e:7} that $0<\eta=s(\proj_{\overline{S}}y)$,
which yields $(x,y)\in\Omega_3$. However, since 
$(x,y)\in\Omega_4$, we have
$\varphi^*(\prox_{\frac{\eta}{\gamma}\varphie}(x/\gamma))\in\RPP$.
Consequently, the claim follows from Lemma~\ref{l:j1}\ref{l:j1-i}. 

Finally, it is clear from \eqref{e:Omega_i} that 
$\Omega_1\cap\Omega_2=\emp$ and $\Omega_1\cap\Omega_3=\emp$.
Moreover, we infer from \ref{p:1iii} and \ref{p:1ii} that
$\Omega_2\cap\Omega_3=\emp$. Altogether, 
$(\Omega_i)_{1\leq i\leq 4}$ is a partition of $\HH\times\GG$ and 
the proof is complete.
\end{proof}

\begin{proposition}
\label{p:1p}
Consider the assumptions and notation of 
Theorem~\ref{t:2}\ref{t:2iii}, and set
\begin{equation}
\label{e:Omega_pi}
\hskip -3.0mm
\begin{cases}
\Xi_1=\mEnge{(u,v)\in\HH\times\GG}
{\haut{s}(\proj_{\cconv S} 
v)=0\;\;\text{and}\;\;
\varphi^*\biggl(\proj_{\cdom\varphi^*}\,
\Big(\dfrac{u}{\gamma}\Big)\biggr)=0}\\[4mm]
\Xi_2=\mEnge{(u,v)\in\HH\times\GG}
{\haut{s}(\proj_{\cconv S} v)\in\RPP
\;\;\text{and}\;
\varphi^*\biggl(\prox_{\frac{1}{\gamma}
\haut{s}(\proj_{\cconv S} v)\varphi^*}\,
\Big(\dfrac{u}{\gamma}\Big)\biggr)=0}\\[4mm]
\Xi_3=\mEnge{(u,v)\in\HH\times\GG}
{\haut{s}\Big(\prox_{\gamma(-\varphi^*(
\proj_{\cdom\varphi^*}(\frac{u}{\gamma})))\haut{s}}v\Big)=0
\;\;\text{and}\;\;\varphi^*\biggl(\proj_{\cdom\varphi^*}\,
\Big(\dfrac{u}{\gamma}\Big)\biggr)<0}\\[4mm]
\Xi_4=(\HH\times\GG)\smallsetminus
(\Xi_1\cup\Xi_2\cup\Xi_3).
\end{cases}
\end{equation}
Then exactly one of the following holds:
\begin{enumerate}
\item 
\label{p:1pi}
$(x,y)\in\Xi_1$, 
$\eta=0$, and $\prox_{\gamma(\varphi\persp s)}(x,y)
=\Big(x-\gamma\,\proj_{\cdom\varphie}
(x/{\gamma}),\proj_{\cconv S}y\Big)$.
\item 
\label{p:1piii}
$(x,y)\in\Xi_2$, $\eta=0$, and 
$\prox_{\gamma(\varphi\persp s)}(x,y)
=\Big(x-\gamma\,\prox_{\frac{1}{\gamma}\haut{s}
(\proj_{\cconv S}y)\varphi^*}
(x/\gamma),
\proj_{\cconv S}y\Big)$.
\item 
\label{p:1pii}
$(x,y)\in\Xi_3$, 
$\eta=-\varphi^*\bigl(\proj_{\cdom\varphie}\,
\bigl(x/\gamma\bigr)\bigr)>0$, and 
\begin{equation}
\prox_{\gamma(\varphi\persp s)}(x,y)=\Big(x-\gamma\,
\proj_{\cdom\varphie}
\Big(\frac{x}{\gamma}\Big),
\prox_{\bigl(-\gamma\varphie\bigl(\proj_{\cdom\varphie}\,
\bigl(\frac{x}{\gamma}\bigr)\bigr)\bigr)\haut{s}}y\Big).
\end{equation}
\item 
\label{p:1piv}
$(x,y)\in\Xi_4$, $\eta>0$ solves 
\begin{equation}
\label{e:mainriiia}
\varphi^*\biggl(\prox_{\frac{1}{\gamma}
\haut{s}(\prox_{\gamma\eta\haut{s}} y)\varphie}\,
\Big(\frac{x}{\gamma}\Big)\biggr)+\eta=0
\end{equation}
and
\begin{equation}
\label{e:Kjiiia}
\prox_{\gamma(\varphi\persp s)}(x,y)
=\biggl(x-\gamma\,\prox_{\frac{1}{\gamma}\haut{s}
(\prox_{\gamma\eta\haut{s}}y)\varphie}
\Big(\frac{x}{\gamma}\Big),
\prox_{\gamma\eta\haut{s}}y\biggr).
\end{equation}
\end{enumerate}
\end{proposition}
\begin{proof} 
It follows from Lemma~\ref{l:24}\ref{l:24v} that 
\begin{equation}
\label{e:p7}
\cdom\haut{s}=\cconv S.
\end{equation}

\ref{p:1pi}:
Since Lemma~\ref{l:j1}\ref{l:j1-i} and \eqref{e:p7} yield
\begin{equation}
\label{e:mainriiip}
\varphi^*\biggl(\prox_{\frac{1}{\gamma}
\haut{s}(\proj_{\cconv S} y)\rocky\varphie}\,
\Big(\frac{x}{\gamma}\Big)\biggr)=
\varphi^*\biggl(\proj_{\cdom\varphie}\,
\Big(\frac{x}{\gamma}\Big)\biggr)=0,
\end{equation}
we deduce from \eqref{e:mainriii} and Lemma~\ref{l:j1}\ref{l:j1-i} 
that $\eta=0$. The claim therefore follows from \eqref{e:Kjiii}.

\ref{p:1piii}: Lemma~\ref{l:j1}\ref{l:j1-i} yields
\begin{equation}
\varphi^*\biggl(\prox_{\frac{1}{\gamma}
\haut{s}(\proj_{\cconv S} y)\rocky\varphie}\,
\Big(\frac{x}{\gamma}\Big)\biggr)=\varphi^*\biggl(
\prox_{\frac{1}{\gamma}\haut{s}(\proj_{\cconv S} y)\varphie}\,
\Big(\frac{x}{\gamma}\Big)\biggr)=0,
\end{equation}
and we deduce from \eqref{e:mainriii} and 
Lemma~\ref{l:j1}\ref{l:j1-i}
that $\eta=0$. Therefore, the claim follows from \eqref{e:Kjiii}.

\ref{p:1pii}: Since 
$\varphi^*(\proj_{\cdom\varphie}(\frac{x}{\gamma}))\in\RMM$, 
Lemma~\ref{l:j1}\ref{l:j1-i} yields
\begin{align}
&\varphi^*\biggl(\prox_{\frac{1}{\gamma}
\haut{s}(\prox_{\gamma(-\varphie(
\proj_{\cdom\varphie}(\frac{x}{\gamma})))\rocky\haut{s}} 
y)\rocky\varphie}\,
\Big(\frac{x}{\gamma}\Big)\biggr)\nonumber\\
&\hspace{51mm}=\varphi^*\biggl(\prox_{\frac{1}{\gamma}
\haut{s}(\prox_{\gamma(-\varphie(
\proj_{\cdom\varphie}(\frac{x}{\gamma})))\haut{s}} 
y)\rocky\varphie}\,
\Big(\frac{x}{\gamma}\Big)\biggr)\nonumber\\
&\hspace{51mm}=\varphi^*\biggl(\proj_{\cdom\varphie}\,
\Big(\frac{x}{\gamma}\Big)\biggr)
\end{align}
and we deduce from \eqref{e:mainriii} that 
$\eta=-\varphi^*(\proj_{\cdom\varphie}(x/\gamma))>0$. Therefore,
the claim follows from \eqref{e:Kjiii}.

\ref{p:1piv}: Suppose that $\eta=0$, hence 
$\haut{s}(\proj_{\cconv S} y)<\pinf$. Then it follows from
\eqref{e:mainriii} that 
\begin{equation}
\label{e:auxp23}
0=\varphi^*\biggl(\prox_{\frac{1}{\gamma}
\haut{s}(\proj_{\cconv S} y)\rocky\varphie}\,
\Big(\frac{x}{\gamma}\Big)\biggr).
\end{equation}
Therefore, if $\haut{s}(\proj_{\cconv S} y)=0$,
then \eqref{e:auxp23} yields 
$0=\varphi^*(\proj_{\cdom\varphie}(x/\gamma))$, which
implies that $(x,y)\in\Xi_1$. On the other hand, if
$\haut{s}(\proj_{\cconv S} y)>0$, then
\eqref{e:auxp23} yields
\begin{equation}
0=\varphi^*\biggl(\prox_{\frac{1}{\gamma}
\haut{s}(\proj_{\cconv S} y)\varphie}\,
\Big(\frac{x}{\gamma}\Big)\biggr)
\end{equation}
and thus $(x,y)\in\Xi_2$. 
At the same time, since $(x,y)\in\Xi_4$, we have
$(x,y)\notin\Xi_1\cup\Xi_2$.
This contradiction shows that $\eta>0$. 
In turn, \eqref{e:mainriii} reduces to 
\begin{equation}
\label{e:mainriiiap}
-\eta=\varphi^*\biggl(\prox_{\frac{1}{\gamma}
\haut{s}(\prox_{\gamma\eta\haut{s}} y)\rocky\varphie}\,
\Big(\frac{x}{\gamma}\Big)\biggr).
\end{equation}
Hence, if 
$\haut{s}(\prox_{\gamma\eta\haut{s}} y)=0$, 
$0>-\eta=\varphi^*(\proj_{\cdom\varphie}(x/\gamma))$,
which yields $(x,y)\in\Xi_3$. However, since 
$(x,y)\in\Xi_4$, we have
$\haut{s}(\prox_{\gamma\eta\haut{s}} y)>0$.
Consequently, the claim follows from Lemma~\ref{l:j1}\ref{l:j1-i}. 

To conclude the proof, we observe that \eqref{e:Omega_pi} yields
$\Xi_1\cap\Xi_2=\emp$ and $\Xi_1\cap\Xi_3=\emp$, and 
we infer from \ref{p:1piii} and \ref{p:1pii} that
$\Xi_2\cap\Xi_3=\emp$. Altogether, 
$(\Xi_i)_{1\leq i\leq 4}$ is a partition of $\HH\times\GG$. 
\end{proof}

\begin{remark}
\label{r:i}
In cases \ref{p:1i}--\ref{p:1ii} of Proposition~\ref{p:1}, the
computation of $\prox_{\gamma(\varphi\persp s)}(x,y)$ requires
only the ability to compute the projection operators onto
${\cdom\varphi^*}$ and ${\overline{S}}$, as well as the proximity
operators of ${\varphi^*}$ and ${\bas{(-s)}}$. Examples of explicit
formulas for these operators can be found in \cite{Livre1,Chie18}.
The case \ref{p:1iv} requires additionally the solution
$\eta\in\RPP$ to \eqref{e:c6}. To determine $\eta$, let us define
$\phi_1$ and $\phi_2$ as in \eqref{e:r-} and note that it is the
root of $T=\phi_1\circ\phi_2+\Id$. Since Lemma~\ref{l:j2} implies
that $T$ is strictly increasing and continuous on $\RPP$, $\eta$ 
can be found via standard one-dimensional root-finding routines
\cite[Chapter~9]{Pres07}. A similar observation can be made for 
Proposition~\ref{p:1p}.
\end{remark}

\section{Examples}
\label{sec:5}
We illustrate the proposed computation of the proximity operator of
perspective functions in the context of applications arising in
mean field type control \cite{Ach16a,Ach16b}, optimal transport
\cite{Dolb09,Papa14}, statistics \cite{Ejst20,Owen07}, and 
thermostatistics \cite{Berc13,Lutw05}. These applications share a
radial base function model, which motivates the following two
examples. 

\begin{example}
\label{ex:28} 
Let $\phi\in\Gamma_0(\RR)$ be an even coercive function such 
that $\phi^*(\RR)\subset\RPX$ and $(\phi^*)^{-1}(\RPP)\neq\emp$, 
and 
set $\varphi=\phi\circ\|\cdot\|$.
Then, $\phi^*\in\Gamma_0(\RR)$ is even, $0\in \inte\dom\phi^*$ by 
\cite[Proposition~14.16]{Livre1},
$\varphi\in\Gamma_0(\HH)$, and
\cite[Example~13.8]{Livre1} implies that 
$\varphi^*=\phi^*\circ\|\cdot\|$.
In turn, $\varphi^*(\HH)\subset\RPX$ and
$(\varphi^*)^{-1}(\RPP)\neq\emp$. 
Furthermore, let $-s\in\Gamma_0(\GG)$, suppose that 
$s^{-1}(\RPP)\neq\emp$, let 
$(x,y)\in\HH\times\GG$, and note that
Proposition~\ref{p:2}\ref{p:2iii} asserts that
\begin{equation}
\label{e:9s5}
(\varphi\persp s)(x,y)=
\begin{cases}
s(y){\phi}\biggl(\dfrac{\|x\|}{s(y)}\biggr),
&\text{if}\;\;0<s(y)<\pinf;\\
(\rec\phi)(\|x\|),&\text{if}\;\;s(y)=0;\\
\pinf,&\text{otherwise}.
\end{cases}
\end{equation}
Now let $\gamma\in\RPP$. It follows from 
Lemma~\ref{l:5}\ref{l:5iii}
and Proposition~\ref{p:1} that the sets 
\begin{equation}
\label{e:Omega_iex}
\begin{cases}
\Omega_1\!=\!\mEnge{(u,v)\in\HH\times\GG}
{\phi^*\biggl(\proj_{\cdom\phietoile}
\biggl(\dfrac{\|u\|}{\gamma}\biggr)\biggr)=0\;\text{and}\;
s(\proj_{\overline{S}}v)=0}\\[4mm]
\Omega_2\!=\!\mEnge{(u,v)\in\HH\times\GG}
{\phi^*\biggl(\proj_{\cdom\phietoile}
\biggl(\dfrac{\|u\|}{\gamma}\biggr)\biggr)\in\RPP
\;\text{and}\;
s\biggl(\prox_{\gamma\phi^*\bigl(\proj_{\cdom\phietoile}
\bigl(\frac{\|u\|}{\gamma}\bigr)\bigr)\bas{(-s)}}\,v\biggr)=0}\\[4mm]
\Omega_3\!=\!\mEnge{(u,v)\in\HH\times\GG}
{\phi^*\biggl(\prox_{\frac{s(\proj_{\overline{S}}v)}{\gamma}
\phietoile}\biggl(\dfrac{\|u\|}{\gamma}\biggr)\biggr)=0
\;\text{and}\;s(\proj_{\overline{S}}v)\in\RPP}\\[4mm]
\Omega_4\!=\!(\HH\times\GG)\smallsetminus
(\Omega_1\cup\Omega_2\cup\Omega_3)
\end{cases}
\end{equation}
form a partition of $\HH\times\GG$,
which brings up four cases for consideration:
\begin{itemize}
\item 
$(x,y)\in\Omega_1$: We derive from Lemma~\ref{l:j1}\ref{l:j1-i}, 
Lemma~\ref{l:5}\ref{l:5ii}, and Proposition~\ref{p:1}\ref{p:1i}
that
\begin{equation}
\label{e:Kjexo1}
\prox_{\gamma(\varphi\persp s)}(x,y)
=
\begin{cases}
\biggl(\biggl(1-\dfrac{\gamma}{\|x\|}\,
\proj_{\cdom\phietoile}
\biggl(\dfrac{\|x\|}{\gamma}\biggr)\biggr)x,\proj_{\overline{S}}\, 
y\biggr),&\text{if}\:\:x\neq0;\\[3mm]
\bigl(0,\proj_{\overline{S}}\, y\bigr),&\text{if}\:\:x=0.
\end{cases}
\end{equation}
\item 
$(x,y)\in\Omega_2$: We derive from 
Lemma~\ref{l:j1}\ref{l:j1-i}, 
Lemma~\ref{l:5}\ref{l:5ii}, and
Proposition~\ref{p:1}\ref{p:1iii} that
\begin{equation}
\label{e:Kjexo2}
\prox_{\gamma(\varphi\persp s)}(x,y)
=
\begin{cases}
\biggl(\biggl(1-\dfrac{\gamma}{\|x\|}\,
\proj_{\cdom\phietoile}
\biggl(\dfrac{\|x\|}{\gamma}\biggr)\biggr)x,
\prox_{\gamma\phietoile\bigl(\proj_{\cdom\phietoile}
\bigl(\frac{\|x\|}{\gamma}\bigr)\bigr)\bas{(-s)}}\,y\biggr),
&\text{if}\:\:x\neq0;\\[4mm]
\Big(0,\prox_{\gamma\phietoile(0)\bas{(-s)}}\,y\Big),
&\text{if}\:\:x=0.
\end{cases}
\end{equation}
\item 
$(x,y)\in\Omega_3$: We derive from 
Lemma~\ref{l:j1}\ref{l:j1-i}, 
Lemma~\ref{l:5}\ref{l:5ii}, and
Proposition~\ref{p:1}\ref{p:1ii} that
\begin{equation}
\label{e:Kjexo3}
\prox_{\gamma(\varphi\persp s)}(x,y)
=
\begin{cases}
\biggl(\biggl(1-\dfrac{\gamma}{\|x\|}\,
\prox_{\frac{s(\proj_{\overline{S}}y)}{\gamma}
\phietoile}\biggl(\dfrac{\|x\|}{\gamma}\biggr)\biggr)x,
\proj_{\overline{S}}\, y\biggr),&\text{if}\:\:x\neq 0;\\[3mm]
\bigl(0,\proj_{\overline{S}}\, y\bigr),&\text{if}\:\:x=0.
\end{cases}
\end{equation}
\item 
$(x,y)\in\Omega_4$: In view of
Lemma~\ref{l:5}\ref{l:5ii} and Lemma~\ref{l:j1}\ref{l:j1-i}, 
Theorem~\ref{t:2}\ref{t:2i} and Proposition~\ref{p:1}\ref{p:1iv} 
guarantee the existence of a unique solution
$\eta\in\RPP$ to
\begin{equation}
\label{e:mainrex2}
\eta=s\biggl(\prox_{\gamma
\phietoile\bigl(\prox_{\frac{\eta}{\gamma}\phietoile}
\bigl(\frac{\|x\|}{\gamma}\bigr)\bigr)\bas{(-s)}}\,y\biggr)
\end{equation}
and \cite[Proposition~24.32]{Livre1} yields
\begin{equation}
\label{e:Kjexo4}
\prox_{\gamma(\varphi\persp s)}(x,y)=
\begin{cases}
\biggl(\biggl(1-\dfrac{\gamma}{\|x\|}\,
\prox_{\frac{\eta}{\gamma}\phietoile}
\biggl(\dfrac{\|x\|}{\gamma}\biggr)\biggr)x,\prox_{\gamma
\phietoile\bigl(\prox_{\frac{\eta}{\gamma}\phietoile}
\bigl(\frac{\|x\|}{\gamma}\bigr)\bigr)\bas{(-s)}}\,y\biggr),
&\text{if}\:\:x\neq 0;\\
\Big(0,\prox_{\gamma\phietoile(0)\bas{(-s)}}\,y\Big),
&\text{if}\:\:x=0.
\end{cases}
\end{equation}
\end{itemize}
\end{example}

Our next example addresses the counterpart of the previous one in
which the sign of $\phi^*$ is flipped.

\begin{example}
\label{ex:28p} 
Let $\phi\in\Gamma_0(\RR)$ be an even coercive function such 
that 
$\phi^*(\RR)\subset\RM\cup\{\pinf\}$ and
$(\phi^*)^{-1}(\RMM)\neq\emp$, and set
$\varphi=\phi\circ\|\cdot\|$. As in Example~\ref{ex:28},
$\phi^*\in\Gamma_0(\RR)$ is even, $0\in \inte\dom\phi^*$,
$\varphi\in\Gamma_0(\HH)$, and
$\varphi^*=\phi^*\circ\|\cdot\|$.
In turn, $\varphi^*(\HH)\subset\RM\cup\{\pinf\}$ and
$(\varphi^*)^{-1}(\RMM)\neq\emp$. In addition, let
$s\in\Gamma_0(\GG)$, suppose that 
$s^{-1}(\RPP)\neq\emp$, and let $(x,y)\in\HH\times\GG$. Then, by
Proposition~\ref{p:2}\ref{p:2i}, 
\begin{equation} 
\label{e:3b5}
(\varphi\persp s)(x,y)=
\begin{cases}
s(y){\phi}
\biggl(\dfrac{\|x\|}{s(y)}\biggr),
&\text{if}\;\;0<s(y)<\pinf;\\
(\rec\phi)(\|x\|),&\text{if}\;\;y\in\cconv 
S\:\:\text{and}\:\:s(y)\leq 0;\\
\pinf,&\text{otherwise}.
\end{cases}
\end{equation} 
Now let $\gamma\in\RPP$. It follows from
Lemma~\ref{l:5}\ref{l:5iii}
and Proposition~\ref{p:1p} that the sets 
\begin{equation}
\label{e:Omega_piex}
\hskip -3.0mm
\begin{cases}
\Xi_1=\mEnge{(u,v)\in\HH\times\GG}
{\haut{s}(\proj_{\cconv S}v)=0\;\;\text{and}\;\;
\phi^*\biggl(\proj_{\cdom\phietoile}\,
\biggl(\dfrac{\|u\|}{\gamma}\biggr)\biggr)=0}\\[4mm]
\Xi_2=\mEnge{(u,v)\in\HH\times\GG}
{\haut{s}(\proj_{\cconv S} v)\in\RPP
\;\;\text{and}\;
\phi^*\biggl(\prox_{\frac{1}{\gamma}
\haut{s}(\proj_{\cconv S} v)\phietoile}\,
\biggl(\dfrac{\|u\|}{\gamma}\biggr)\biggr)=0}\\[4mm]
\Xi_3=\mEnge{(u,v)\in\HH\times\GG}
{\haut{s}\Big(\prox_{\gamma\bigl(-\phietoile\bigl(
\proj_{\cdom\phietoile}\bigl(\frac{\|u\|}{\gamma}\bigr)\bigr)\bigr)
\haut{s}}v\Big)=0
\;\;\text{and}\;\;\phi^*\biggl(\proj_{\cdom\phietoile}\,
\biggl(\dfrac{\|u\|}{\gamma}\biggr)\biggr)<0}\\[4mm]
\Xi_4=(\HH\times\GG)\smallsetminus
(\Xi_1\cup\Xi_2\cup\Xi_3)
\end{cases}
\end{equation}
form a partition of $\HH\times\GG$. This leads us to 
consider the following cases:
\begin{itemize}
\item 
$(x,y)\in\Xi_1$: We derive from 
Lemma~\ref{l:j1}\ref{l:j1-i}, 
Lemma~\ref{l:5}\ref{l:5ii}, and
Proposition~\ref{p:1p}\ref{p:1pi} that
\begin{equation}
\label{e:Kjexo1p}
\prox_{\gamma(\varphi\persp s)}(x,y)
=
\begin{cases}
\biggl(\biggl(1-\dfrac{\gamma}{\|x\|}\,
\proj_{\cdom\phietoile}
\biggl(\dfrac{\|x\|}{\gamma}\biggr)\biggr)x,\proj_{\cconv S}\, 
y\biggr),&\text{if}\:\:x\neq0;\\
\bigl(0,\proj_{\cconv S}\, y\bigr),&\text{if}\:\:x=0.
\end{cases}
\end{equation}
\item 
$(x,y)\in\Xi_2$: We derive from 
Lemma~\ref{l:j1}\ref{l:j1-i}, 
Lemma~\ref{l:5}\ref{l:5ii}, and
Proposition~\ref{p:1p}\ref{p:1piii} that
\begin{equation}
\label{e:Kjexo2p}
\prox_{\gamma(\varphi\persp s)}(x,y)
=
\begin{cases}
\biggl(\biggl(1-\dfrac{\gamma}{\|x\|}\,
\prox_{\frac{1}{\gamma}\haut{s}
(\proj_{\cconv S}y)\phietoile}
\biggl(\dfrac{\|x\|}{\gamma}\biggr)\biggr)x,
\proj_{\cconv S}\, y\biggr),
&\text{if}\:\:x\neq0;\\[4mm]
\bigl(0,\proj_{\cconv S}y\bigr),
&\text{if}\:\:x=0.
\end{cases}
\end{equation}
\item 
$(x,y)\in\Xi_3$: We derive from 
Lemma~\ref{l:j1}\ref{l:j1-i}, 
Lemma~\ref{l:5}\ref{l:5ii}, and
Proposition~\ref{p:1p}\ref{p:1pii} that
\begin{equation}
\label{e:Kjexo3p}
\prox_{\gamma(\varphi\persp s)}(x,y)=
\begin{cases}
\biggl(\biggl(1-\dfrac{\gamma}{\|x\|}\,
\proj_{\cdom\phietoile}
\biggl(\dfrac{\|x\|}{\gamma}\biggr)\biggr)x,
\prox_{\bigl(-\gamma\phietoile\bigl(\proj_{\cdom\phietoile}
\bigl(\frac{\|x\|}{\gamma}\bigr)\bigr)\bigr)\haut{s}}y\biggr),
&\text{if}\:\:x\neq0;\\
\bigl(0,\prox_{(-\gamma\phietoile(0))\haut{s}}y\bigr),
&\text{if}\:\:x=0.
\end{cases}
\end{equation}
\item 
$(x,y)\in\Xi_4$: By virtue of Lemma~\ref{l:5}\ref{l:5ii} and
Lemma~\ref{l:j1}\ref{l:j1-i}, it follows from 
Theorem~\ref{t:2}\ref{t:2iii} and
Proposition~\ref{p:1p}\ref{p:1piv} that there exists a 
unique solution $\eta\in\RPP$ to
\begin{equation}
\label{e:mainrex2p}
\eta=-\phi^*\biggl(\prox_{\frac{1}{\gamma}
\haut{s}(\prox_{\gamma\eta\haut{s}} y)\phietoile}\,
\biggl(\frac{\|x\|}{\gamma}\biggr)\biggr)
\end{equation}
and, hence, \cite[Proposition~24.32]{Livre1} yields
\begin{equation}
\label{e:Kjexo4p2}
\prox_{\gamma(\varphi\persp s)}(x,y)=
\begin{cases}
\biggl(\biggl(1-\dfrac{\gamma}{\|x\|}\,
\prox_{\frac{1}{\gamma}\haut{s}
(\prox_{\gamma\eta\haut{s}}y)\phietoile}
\biggl(\dfrac{\|x\|}{\gamma}\biggr)\biggr)x,
\prox_{\gamma\eta\haut{s}}y\biggr),
&\text{if}\:\:x\neq 0;\\
\bigl(0,\prox_{(-\gamma\phietoile(0))\haut{s}}y\bigr),
&\text{if}\:\:x=0.
\end{cases}
\end{equation}
\end{itemize}
\end{example}

Starting with the work \cite{Bren00}, convex optimization problems
involving the perspective function with linear scaling~\eqref{e:0}
appear in optimal transport theory and in mean field games
\cite{Achd12,Bena15,Lasr07,Vill03}. In this context, numerical
methods employing its proximity operator are investigated in
\cite{Bric18}. Extensions to variational models with $q$th root
scaling functions have been proposed to address optimal control of
McKean--Vlasov systems with congestion \cite{Ach16a,Ach16b}, as
well as optimal transport with nonlinear mobilities \cite{Dolb09}.
In the following example, we compute the proximity operator of
perspective functions with such scaling functions and incorporate a
scale constraint which can be used, in particular, to model density
constraints \cite{Card16,Daud22,Daud23,Mesz15}.

\begin{example}
\label{ex:4}
Let $p\in\left]1,\pinf\right[$ and $q\in\zeroun$.
Set $p^*=p/(p-1)$, and $\phi=|\cdot|^p/p$. Let $I\subset\RP$ be a
closed interval with nonempty interior such that $0\in I$ and
define
\begin{equation}
\label{e:v1}
\psi\colon\RR\to\{\minf\}\cup\RP\colon y\mapsto 
\begin{cases}
y^q,&\text{if}\:\:y\geq 0;\\
\minf,&\text{if}\:\:y<0,
\end{cases}
\quad\text{and set}\quad s=\psi-\iota_{I}.
\end{equation}
Then $\phi^*=|\cdot|^{\petoile}/p^*$, $\bas{(-s)}=-s$,
$\dom\phi^*=\RR$, and $\overline{S}=I$. This places us 
in the framework of Example~\ref{ex:28} with $\GG=\RR$,
\eqref{e:9s5} reduces to
\begin{equation}
\label{e:9s5ex}
(\varphi\persp s)(x,y)=
\begin{cases}
\dfrac{\|x\|^p}{p\,y^{q(p-1)}},
&\text{if}\;\;0<y\in I;\\
0,&\text{if}\;\;x=0\:\:\text{and}\:\:y=0;\\
\pinf,&\text{otherwise},
\end{cases}
\end{equation}
and \eqref{e:Omega_iex} reduces to
\begin{equation}
\label{e:Omega_iexp}
\hskip -3.0mm
\begin{cases}
\Omega_1=\{0\}\times\RM\\
\Omega_2=\mEnge{(u,v)\in\HH\times\RR}
{u\neq 0\;\;\text{and}\;
\prox_{\frac{\gamma}{\petoile}
\big|\frac{\|u\|}{\gamma}\big|^{\petoile}(-s)}\,v=0}\\[1mm]
\Omega_3=\{0\}\times\RPP\\
\Omega_4=\bigl((\HH\smallsetminus\{0\})\times\RR\bigr)
\smallsetminus\Omega_2.
\end{cases}
\end{equation}
Note that, for every $\mu\in\RPP$ and every $v\in\RR$, 
Lemma~\ref{l:j1}\ref{l:j1iii} yields
$\prox_{-\mu s}v\in\dom(-s)=\RP$. Moreover, if $\prox_{-\mu s}v=0$,
Lemma~\ref{l:j1}\ref{l:j1iii} implies
\begin{equation}
(\forall y\in\RP)\quad y^q(y^{1-q}v+\mu)\leq 0,
\end{equation}
which is not possible.
Hence, $\prox_{-\mu s}v\in\RPP$ and we deduce that $\Omega_2=\emp$ 
and $\Omega_4=(\HH\smallsetminus\{0\})\times\RR$.
Therefore, Example~\ref{ex:28} and 
\cite[Proposition~24.47]{Livre1} yield
\begin{equation}
\label{e:y91}
\prox_{\gamma(\varphi\persp s)}(x,y)=
\begin{cases} 
\bigl(0,\proj_Iy\bigr),&\text{if}\;\;x=0;\\[2mm]
\biggl(\biggl(1-\dfrac{\gamma}{\|x\|}\,
\prox_{\frac{\eta}{\gamma}\phietoile}
\biggl(\dfrac{\|x\|}{\gamma}\biggr)\biggr)x,
\proj_{I}\biggl(\prox_{\frac{\gamma}{\petoile}
\big|\prox_{\frac{\eta}{\gamma}\phietoile}
\bigl(\frac{\|x\|}{\gamma}\bigr)\big|^{\petoile}(-\psi)}\, 
y\biggr)\biggr),&\text{if}\;\;x\neq 0,
\end{cases}
\end{equation}
where, if $x\neq 0$, $\eta$ is the unique solution in $\RPP$ to
\begin{equation}
\label{e:7y1}
\eta=\bigg|\proj_I\Big(\prox_{\frac{\gamma}{\petoile}
\big|\prox_{\frac{\eta}{\gamma}\phietoile}
\bigl(\frac{\|x\|}{\gamma}\bigr)\big|^{\petoile}(-\psi)}\, 
y\Big)\bigg|^q.
\end{equation}
Note that, in view of \eqref{e:7y1}, \eqref{e:y91} can be written
as
\begin{equation}
\label{e:y81}
\prox_{\gamma(\varphi\persp s)}(x,y)=
\begin{cases} 
\bigl(0,\proj_Iy\bigr),&\text{if}\;\;x=0;\\[2mm]
\biggl(\biggl(1-\dfrac{\gamma}{\|x\|}\,
\prox_{\frac{\eta}{\gamma}\phietoile}
\biggl(\dfrac{\|x\|}{\gamma}\biggr)\biggr)x,\eta^{1/q}\biggr),
&\text{if}\;\;x\neq 0.
\end{cases}
\end{equation}
In the case when $x\neq 0$, let us point out that, given
$\xi\in\RPP$, \cite[Example~24.38]{Livre1} asserts that 
$\rho(\xi)=\prox_{\frac{\xi}{\gamma}\phietoile}
({\|x\|}/{\gamma})$ is the unique solution to 
\begin{equation}
\label{e:proxphi}
\|x\|=\rho\gamma+\xi\rho^{\,\petoile-1}.
\end{equation}
On the other hand, for every $\mu\in\RPP$, in view of 
Lemma~\ref{l:j1}\ref{l:j1iii},
$z(\mu)=\prox_{\gamma\mu(-\psi)}\,y\in\RPP$
is the unique solution to 
\begin{equation}
y=z-q\gamma\mu z^{q-1}.
\end{equation}
Therefore, finding $\eta\in\RPP$ such that \eqref{e:7y1} holds
amounts to solving 
$\eta=|\proj_I(z(\rho(\eta)^{\petoile}/p^*))|^q$, that is,
\begin{equation}
\eta=\bigg|\min\bigg\{z\biggl(\dfrac{\rho(\eta)^{\petoile}}
{p^*}\biggr),
\sup I\bigg\}\bigg|^q,
\end{equation}
which can be handled by one-dimensional root-finding methods.
\end{example}

As mentioned in Section~\ref{sec:1}, problems in statistical
inference and in robust statistics involve joint location/scale
estimation; see \cite{Anto10,Ejst20,Stat21,Lamb16,Owen07} for
further instances of this model. Most of these models involve the
perspective function of the Huber function with a scalar scale. Our
analysis allows us to extend it to nonlinear scales. An
illustration of Example~\ref{ex:28p} in this context is provided in
the following example, where the proximity operator of the
resulting function, a central piece in algorithms for solving
concomitant estimation problems \cite{Jmaa18,Ejst20,Stat21}, is
computed.

\begin{example}
\label{ex:5}
Let $\phi$ is the Huber function with parameter $\alpha\in\RPP$,
that is,
\begin{equation}
\phi\colon\xi\mapsto
\begin{cases}
\alpha|\xi|,&\;\;\text{if}\:\:|\xi|>\alpha;\\[2mm]
\dfrac{|\xi|^2+\alpha^2}{2},&\;\;
\text{if}\:\:|\xi|\leq\alpha.
\end{cases}
\end{equation}
It follows from \cite[Example~13.7]{Livre1} that
\begin{equation}
\phi^*\colon\xi^*\mapsto 
\begin{cases}\pinf, 
&\text{if}\:\:|\xi^*|>\alpha;\\[2mm]
\dfrac{|\xi^*|^{2}-\alpha^{2}}{2},
&\text{if}\:\:|\xi^*|\leq\alpha.
\end{cases}
\end{equation}
Therefore $\dom\phi^*=[-\alpha,\alpha]$ and we deduce from 
\cite[Example~24.9]{Livre1} and Lemma~\ref{l:j1}\ref{l:j1ii} that
\begin{equation}
\label{e:je}
\prox_{\gamma\phietoile}\colon \xi\mapsto 
\begin{cases}
\alpha\,\sign(\xi),&\text{if}\:\:|\xi|>(\gamma+1)\alpha;\\
\xi/(\gamma+1),&\text{if}\:\:|\xi|\leq(\gamma+1)\alpha.
\end{cases}
\end{equation}
Furthermore, let $\beta\in\RPP$ and set
\begin{equation}
\label{e:v15}
s\colon\RR\to\RP\colon y\mapsto 
\sqrt{\beta+|y|^2}.
\end{equation}
Then $\haut{s}=s$ and $\cconv{S}=\RR$. Altogether, we are
in the framework of Example~\ref{ex:28p} with $\GG=\RR$,
\eqref{e:3b5} reduces to
\begin{equation} 
\label{e:3b5ex}
(\varphi\persp s)(x,y)=
\begin{cases}
\alpha\|x\|,&\text{if}\:\: \|x\|>\alpha \sqrt{\beta+|y|^2};\\[2mm]
\dfrac{\|x\|^2+\alpha^2(\beta+y^2)}
{2\sqrt{\beta+|y|^2}}&\text{if}\:\: 
\|x\|\leq\alpha \sqrt{\beta+|y|^2},
\end{cases}
\end{equation} 
and \eqref{e:Omega_piex} reduces to
\begin{equation}
\label{e:Omega_iexp5}
\hskip -3.0mm
\begin{cases}
\Xi_1=\Xi_3=\emp\\
\Xi_2=\menge{(u,v)\in\HH\times\RR}
{\|u\|\geq\alpha\bigl(\sqrt{\beta+v^2}+\gamma\bigr)}\\
\Xi_4=\menge{(u,v)\in\HH\times\RR}
{\|u\|<\alpha\bigl(\sqrt{\beta+v^2}+\gamma\bigr)}.
\end{cases}
\end{equation}
Therefore, we deduce from Example~\ref{ex:28p} that 
\begin{align}
\prox_{\gamma(\varphi\persp s)}(x,y)
&=
\begin{cases}
\biggl(\biggl(1-\dfrac{\alpha\gamma}{\|x\|}\biggr)x,y\biggr),
&\text{if}\:\:\|x\|\geq\alpha\bigl(\sqrt{\beta+y^2}+\gamma\bigr);
\\[4mm]
\biggl(\biggl(\dfrac{\sqrt{\beta+q(\eta,y)^2}}
{\gamma+\sqrt{\beta+q(\eta,y)^2}}\biggr)x,q(\eta,y)\biggr),
&\text{if}\:\:\|x\|<\alpha\bigl(\sqrt{\beta+y^2}+\gamma\bigr),
\end{cases}
\end{align}
where $\eta$ is the unique solution in $\RPP$ to
\begin{equation}
\label{e:88}
\eta=\dfrac{\alpha^2|\gamma+\sqrt{\beta+q(\eta,y)^2}|^2-\|x\|^2}
{2|\gamma+\sqrt{\beta+q(\eta,y)^2}|^2}
\end{equation}
and $q(\eta,y)$ is the unique solution to the quartic equation
\begin{equation}
\label{e:qf}
q^4-2yq^3+(y^2+\beta-\gamma^2\eta^2)q^2-2\beta yq+\beta y^2=0,
\end{equation}
in $[0,y]$ if $y\geq 0$ and in $[y,0]$ if $y<0$. In view of 
\cite[Section~2.2.3]{Zwil18}, \eqref{e:qf} can be solved 
explictly. We then solve \eqref{e:88} via one-dimensional
root-finding algorithms \cite[Chapter~9]{Pres07}.
\end{example}


\begin{thebibliography}{99} 
\setlength{\itemsep}{-1pt}
\small

\bibitem{Achd12} 
Y. Achdou, F. Camilli, and I. Capuzzo--Dolcetta, 
Mean field games: Numerical methods for the planning problem,
{\em SIAM J. Control Optim.},
vol. 50, pp. 77--109, 2012.

\bibitem{Ach16a} 
Y. Achdou and M. Lauri\`ere,
Mean field type control with congestion,
{\em Appl. Math. Optim.},
vol. 73, pp. 393--418, 2016.

\bibitem{Ach16b} 
Y. Achdou and M. Lauri\`ere,
Mean field type control with congestion (II): 
An augmented Lagrangian method,
{\em Appl. Math. Optim.},
vol. 74, pp. 535--578, 2016.

\bibitem{Anto10} 
A. Antoniadis,
Comments on: $\ell_1$-penalization for mixture regression models,
{\em TEST},
vol. 19, pp. 257--258, 2010.

\bibitem{Arav18} 
A. Y. Aravkin, J. V. Burke, D. Drusvyatskiy, M. P. Friedlander, 
and K. J. MacPhee,
Foundations of gauge and perspective duality,
{\em SIAM J. Optim.},
vol. 28, pp. 2406--2434, 2018.

\bibitem{Atto84} 
H. Attouch,
{\em Variational Convergence for Functions and Operators}.
Pitman, Boston, MA, 1984.

\bibitem{Aval18} 
M. Avalos-Fernandez, R. Nock, C. S. Ong, J. Rouar, and K. Sun,
Representation learning of compositional data,
{\em Adv. Neural Information Process. Syst.},
vol. 31, pp. 6680--6690, 2018.

\bibitem{Livre1} 
H. H. Bauschke and P. L. Combettes, 
{\em Convex Analysis and Monotone Operator Theory in Hilbert 
Spaces}, 2nd ed. 
Springer, New York, 2017.

\bibitem{Bren00}
J.-D. Benamou and Y. Brenier,
A computational fluid mechanics solution to the Monge-Kantorovich
mass transfer problem,
{\em Numer. Math.},
vol. 84, pp. 375--393, 2000.

\bibitem{Bena15} 
J.-D. Benamou and G. Carlier, 
Augmented Lagrangian methods for transport optimization, 
mean field games and degenerate elliptic equations, 
{\em J. Optim. Theory Appl.}, 
vol. 167, pp. 1--26, 2015.

\bibitem{Berc13} 
J.-F. Bercher, 
Some properties of generalized Fisher information in 
the context of nonextensive thermostatistics,
{\em Physica A}, vol. 392, pp. 3140--3154, 2013. 

\bibitem{Bren02}
Y. Brenier and M. Puel, 
Optimal multiphase transportation with prescribed momentum, 
{\em ESAIM Control Optim. Calc. Var.}, 
vol. 8, pp. 287--343, 2002.

\bibitem{Nmtm09}
L. M. Brice\~no-Arias and P. L. Combettes,
Convex variational formulation with smooth coupling for 
multicomponent signal decomposition and recovery, 
{\em Numer. Math. Theory Methods Appl.},
vol. 2, pp. 485--508, 2009.

\bibitem{Vina23}
L. M. Brice\~{n}o-Arias, P. L. Combettes, and F. J. Silva,
Perspective functions with nonlinear scaling, 
{\em Commun. Contemp. Math.},
published online 2024-02-19.

\bibitem{Bric18}
L. M. Brice\~no-Arias, D. Kalise, and F. J. Silva, 
Proximal methods for stationary mean field games with local
couplings, 
{\em SIAM J. Control Optim.}, 
vol. 56, pp. 801--836, 2018. 

\bibitem{Card13}
P. Cardaliaguet, G. Carlier, and B. Nazaret,
Geodesics for a class of distances in the space of probability
measures,
{\em Calc. Var. Partial Differential Equations},
vol. 48, pp. 395--420, 2013.

\bibitem{Card16}
P. Cardaliaguet, A. R. M\'esz\'aros, and F. Santambrogio,
First order mean field games with density constraints: Pressure
equals price,
{\em SIAM J. Control Optim.},
vol. 54, pp. 2672--2709, 2016.

\bibitem{Carl19} 
E. A. Carlen and E. H. Lieb,
Some trace inequalities for exponential and logarithmic functions,
{\em Bull. Math. Sci.}, 
vol. 9, art. 1950008, 2019. 

\bibitem{Car22b} 
J. A. Carrillo, M. G. Delgadino, and J. Wu,
Boltzmann to Landau from the gradient flow perspective,
{\em Nonlinear Anal.},
vol. 219, art. 112824, 2022.

\bibitem{Carr10}
J. A. Carrillo, S. Lisini, G. Savar\'e, and D. Slep\v{c}ev,
Nonlinear mobility continuity equations and generalized 
displacement convexity, 
{\em J. Funct. Anal.},
vol. 258, pp. 1273--1309, 2010.

\bibitem{Carr24}
J. A. Carrillo, L. Wang, and C. Wei, 
Structure preserving primal dual methods for gradient flows with 
nonlinear mobility transport distances, 
{\em SIAM J. Numer. Anal.},
vol. 62, pp. 376--399, 2024.

\bibitem{Cham16}
A. Chambolle and T. Pock,
An introduction to continuous optimization for imaging,
{\em Acta Numer.},
vol. 25, pp. 161--319, 2016. 

\bibitem{Chie18}
G. Chierchia, E. Chouzenoux, P. L. Combettes, and J.-C. Pesquet,
The proximity operator repository. 
\url{http://proximity-operator.net/}

\bibitem{Svva18}
P. L. Combettes,
Perspective functions: Properties, constructions, and examples,
{\em Set-Valued Var. Anal.},
vol. 26, pp. 247--264, 2018.

\bibitem{Acnu24} 
P. L. Combettes, 
The geometry of monotone operator splitting methods,
{\em Acta Numer.}, 
vol. 33, pp. 487--632, 2024.

\bibitem{Jmaa18}
P. L. Combettes and C. L. M\"uller, 
Perspective functions: Proximal calculus and applications 
in high-dimensional statistics,
{\em J. Math. Anal. Appl.},
vol. 457, pp. 1283--1306, 2018.

\bibitem{Ejst20} 
P. L. Combettes and C. L. M\"uller, 
Perspective maximum likelihood-type estimation via proximal
decomposition,
{\em Electron. J. Stat.},
vol. 14, pp. 207--238, 2020.

\bibitem{Stat21} 
P. L. Combettes and C. L. M\"uller, 
Regression models for compositional data: 
General log-contrast formulations, proximal optimization, 
and microbiome data applications,
{\em Stat. Biosciences},
vol. 13, pp. 217--242, 2021.

\bibitem{Cond23}
L. Condat, D. Kitahara, A. Contreras, and A. Hirabayashi,
Proximal splitting algorithms for convex optimization: A tour of 
recent advances, with new twists,
{\em SIAM Rev.},
vol. 65, pp. 375--435, 2023. 

\bibitem{Daud22} 
S. Daudin, 
Optimal control of diffusion processes with terminal constraint in
law,
{\em J. Optim. Theory Appl.},
vol. 195, pp. 1--41, 2022.

\bibitem{Daud23} 
S. Daudin, 
Optimal control of the Fokker--Planck equation under state
constraints in the Wasserstein space,
{\em J. Math. Pures Appl.},
vol. 175, pp. 37--75, 2023.

\bibitem{Dolb09} 
J. Dolbeault, B. Nazaret, and G. Savar\'e,
A new class of transport distances between measures,
{\em Calc. Var. Partial Differential Equations},
vol. 34, pp. 193--231, 2009.

\bibitem{Effr09}
E. Effros, 
A matrix convexity approach to some celebrated quantum
inequalities, 
{\em Proc. Natl. Acad. Sci. USA},
vol. 106, pp. 1006--1008, 2009.

\bibitem{Elgh18}
M. El Gheche, G. Chierchia, and J.-C. Pesquet, 
Proximity operators of discrete information divergences,
{\em IEEE Trans. Inform. Theory},
vol. 64, pp. 1092--1104, 2018.

\bibitem{Hube81}
P. J. Huber, 
{\em Robust Statistics}, 1st ed.
Wiley, New York, 1981.

\bibitem{Imai88} 
H. Imai,
On the convexity of the multiplicative version of Karmarkar's
potential function,
{\em Math. Programming},
vol. 40, pp. 29--32, 1988.

\bibitem{Kuro22}
H. Kuroda and D. Kitahara,
Block-sparse recovery with optimal block partition,
{\em IEEE Trans. Signal Process.},
vol. 70, pp. 1506--1520, 2022.

\bibitem{Lamb16}
S. Lambert-Lacroix and L. Zwald,
The adaptive BerHu penalty in robust regression,
{\em J. Nonparametr. Stat.},
vol. 28, pp. 487--514, 2016.

\bibitem{Lasr07}
J.-M. Lasry and P.-L. Lions, 
Mean field games,
{\em Jpn. J. Math.},
vol. 2, pp. 229--260, 2007.

\bibitem{Lutw05}
E. Lutwak, D. Yang, and G. Zhang,
Cram\'er--Rao and moment-entropy inequalities for
R\'enyi entropy and generalized Fisher information,
{\em IEEE Trans. Inform. Theory}, 
vol. 51, pp. 473--478, 2005.

\bibitem{Maas11}
J. Maas, 
Gradient flows of the entropy for finite Markov chains,
{\em J. Funct. Anal.}, 
vol. 261, pp. 2250--2292, 2011.

\bibitem{Mare01}
P. Mar\'echal,
On the convexity of the multiplicative potential
and penalty functions and related topics,
{\em Math. Program. A},
vol. 89, pp. 505--516, 2001.

\bibitem{Mar05a} 
P. Mar\'echal,
On a functional operation generating convex functions,
Part 1: Duality,
{\em J. Optim. Theory Appl.},
vol. 126, pp. 175--189, 2005.

\bibitem{Mar05b} 
P. Mar\'echal,
On a functional operation generating convex functions,
Part 2: Algebraic properties,
{\em J. Optim. Theory Appl.},
vol. 126, pp. 357--366, 2005.

\bibitem{Mesz15} 
A. R. M\'esz\'aros and F. J. Silva, 
A variational approach to second order mean field games with
density constraints: The stationary case,
{\em J. Math. Pures Appl.},
vol. 104, pp. 1135--1159, 2015.

\bibitem{Mor62b} 
J. J. Moreau, 
Fonctions convexes duales et points proximaux dans un 
espace hilbertien,
{\em C. R. Acad. Sci. Paris},
vol. A255, pp. 2897--2899, 1962.

\bibitem{More65} 
J. J. Moreau,
Proximit\'e et dualit\'e dans un espace hilbertien,
{\em Bull. Soc. Math. France},
vol. 93, pp. 273--299, 1965.

\bibitem{Owen07} 
A. B. Owen, 
A robust hybrid of lasso and ridge regression,
{\em Contemp. Math.},
vol. 443, pp. 59--71, 2007.

\bibitem{Papa14} 
N. Papadakis, G. Peyr\'e, and E. Oudet,
Optimal transport with proximal splitting,
{\em SIAM J. Imaging Sci.}, 
vol. 7, pp. 212--238, 2014.

\bibitem{Pres07} 
W. H. Press, S. A. Teukolsky, W. T. Vetterling, and B. P. Flannery,
{\em Numerical Recipes -- The Art of Scientific Computing},
3rd. ed. Cambridge University Press, Cambridge, MA, 2007.

\bibitem{Rock66}
R. T. Rockafellar,
Level sets and continuity of conjugate convex functions,
{\em Trans. Amer. Math. Soc.}, 
vol. 123, pp. 46--63, 1966.

\bibitem{Rock70} 
R. T. Rockafellar, 
{\em Convex Analysis}.
Princeton University Press, Princeton, NJ, 1970.

\bibitem{Vill03}
C. Villani, 
{\em Topics in Optimal Transportation}.
American Mathematical Society, Providence, RI, 2003.

\bibitem{Yama22} 
I. Yamada and M. Yamagishi,
Hierarchical convex optimization by the hybrid steepest descent
method with proximal splitting operators -- Enhancements of SVM and
Lasso,
in: H. H. Bauschke, R. S. Burachik, D. R. Luke (eds.), 
{\em Splitting Algorithms, Modern Operator Theory, and 
Applications}, pp. 413--489. Springer, New York, 2019.

\bibitem{Zali08} 
C. Z\u{a}linescu, 
On the second conjugate of several convex functions in general
normed vector spaces,
{\em J. Global. Optim.},
vol. 40, pp. 475--487, 2008.

\bibitem{Zell66}
A. Zellner, J. Kmenta, and J. Dr\`eze,
Specification and estimation of Cobb-Douglas production function 
models,
{\em Econometrica}, 
vol. 34, pp. 784--795, 1966.

\bibitem{Berc21}
S. Zozor and J.-F. Bercher, 
$\phi$-informational measures: Some results and interrelations, 
{\em Entropy}, 
vol. 23, art. 911, 2021. 

\bibitem{Zwil18}
D. Zwillinger, ed.,
{\em CRC Standard Mathematical Tables and Formulas},
33rd ed. CRC Press, Boca Raton, FL, 2018.

\end{thebibliography}
\end{document}